\documentclass[10pt]{article}

\usepackage{amsmath}
\usepackage{amsthm}
\usepackage{amssymb}
\usepackage{amsfonts}
\usepackage{graphicx}
\usepackage{verbatim}
\usepackage{mathrsfs}
\usepackage{subfigure}
\usepackage{fancyhdr}
\usepackage{latexsym}
\usepackage{graphicx}
\usepackage{dsfont}

\theoremstyle{plain}
\newtheorem{theorem}{Theorem}[section]

\newtheorem{lemma}[theorem]{Lemma}

\theoremstyle{definition}
\newtheorem{definition}[theorem]{Definition}

\newcommand{\R}{\mathbb{R}}

\newcommand{\bx}{\bar{x}}

\newcommand{\bydef}{:=} 

\title{Rigorous numerics in Floquet theory: \\ computing stable and unstable bundles of periodic orbits}

\author{Roberto Castelli \thanks{{\bf Corresponding author}. BCAM - Basque Center for Applied Mathematics, Bizkaia Technology Park, 48160 Derio, Bizkaia, Spain.
Phone: (+34) 946 567 842. Fax: (+34) 946 567 843. Email: {\tt rcastelli@bcamath.org}.}
\and
\and Jean-Philippe Lessard\thanks {
BCAM - Basque Center for Applied Mathematics, Bizkaia Technology Park, 48160 Derio, Bizkaia, Spain
and Universit\'e Laval, D\'epartement de Math\'ematiques et de Statistique, Pavillon Alexandre-Vachon, 1045 avenue de la M\'edecine, Local 1056, Qu\'ebec, (Qu\'ebec), G1V 0A6, Canada.
Email: {\tt lessard@bcamath.org}.
}}

\date{}

\begin{document}

\maketitle

\begin{abstract}
In this paper, a new rigorous numerical method to compute fundamental matrix solutions of non-autonomous linear differential equations with periodic coefficients is introduced. 
Decomposing the fundamental matrix solutions $\Phi(t)$ by their Floquet normal forms, that is as product of real periodic and exponential matrices $\Phi(t)=Q(t)e^{Rt}$, one solves 
simultaneously for $R$ and for the Fourier coefficients of $Q$ via a fixed point argument in a suitable Banach space of rapidly decaying coefficients. As an application, the method 
is used to compute rigorously stable and unstable bundles of periodic orbits of vector fields.  Examples are given in the context of the Lorenz equations and the $\zeta^3$-model.
\end{abstract}

\begin{center}
{\bf \small Keywords} \\ \vspace{.05cm}
{ \small Rigorous numerics $\cdot$ Floquet theory $\cdot$ Fundamental matrix solutions $\cdot$ \\ Contraction mapping theorem $\cdot$ Periodic orbits $\cdot$ Tangent bundles }
\end{center}

\begin{center}
{\bf \small Mathematics Subject Classification (2000)} \\ \vspace{.05cm}
{ \small 37B55 $\cdot$ 37M99 $\cdot$ 37C27 $\cdot$ 65G99 $\cdot$ 34D05 }
\end{center}

\section{Introduction}
\label{sec:Intro}

In his seminal work \cite{MR1508722} of 1883, Gaston Floquet studied linear non-autonomous differential equations of the form
\begin{equation}\label{eq:syst-A}
\dot y=A(t)y ,
\end{equation}
where $A(t)$ is a $\tau$-periodic continuous matrix function of $t$. The main result of \cite{MR1508722} is now presented, and its proof can be found for instance in \cite{MR2224508}. 
\begin{theorem}{\em \bf [Floquet, 1883]} 
\label{Th-Floquet}
Let $A(t)$ be a $\tau$-periodic continuous matrix function and denote by $\Phi(t)$ a fundamental matrix solution of \eqref{eq:syst-A}. 
Then $\Phi(t+\tau)$ is also a fundamental matrix solution, $\Phi(t+\tau)=\Phi(t)\Phi^{-1}(0)\Phi(\tau)$, and there exist a real constant matrix $R$ and a real nonsingular, continuously differentiable, $2\tau$-periodic matrix function $Q(t)$ such that
\begin{equation} \label{eq:Floquet_normal_form}
\Phi(t)=Q(t) e^{Rt}.
\end{equation}
\end{theorem}  
Decomposition \eqref{eq:Floquet_normal_form} is called a {\em Floquet normal form} for the fundamental matrix solution $\Phi(t)$. The real time-dependent change of coordinates $z = Q^{-1}(t)y$ transforms system \eqref{eq:syst-A} into a linear constant coefficients system of the form $\dot z= R z$. A stability theorem demonstrates that the stability of the zero solution of \eqref{eq:syst-A} can be determined by the eigenvalues of the so-called {\em monodromy matrix} $\Phi(\tau)$. As mentioned in \cite{MR2224508}, while the stability theorem is very elegant, in applied problems it is usually impossible to compute the eigenvalues of the monodromy matrix. An even more challenging and central problem is the computation of the fundamental matrix solutions. The goal of the present work is to address this major difficulty by introducing a new rigorous numerical method to compute explicitly Floquet normal forms as in \eqref{eq:Floquet_normal_form}, hence providing a direct way to obtain fundamental matrix solutions of \eqref{eq:syst-A}.  Before proceeding with a general presentation of the rigorous computational method, let us introduce some motivations. 

First of all, we are not aware of any method to construct rigorously Floquet normal forms as introduced in Theorem~\ref{Th-Floquet}. Since this fundamental decomposition was introduced more than 125 years ago, we believe that developing a rigorous computational method leading to an explicit construction of Floquet normal forms is an important contribution to the field of differential equations. 

The second motivation is directly linked to the study of dynamical systems. Indeed, equations of the form \eqref{eq:syst-A} arise naturally when studying stability properties of time-periodic solutions of differential equations $\dot y=g(y)$, where $g:\R^n \rightarrow \R^n$ is a smooth map. Assume that $\Gamma$ is a $\tau$-periodic orbit of $\dot y=g(y)$ parameterized by $\gamma(t) \in \R^n$ ($t \in [0,\tau]$), and define the $\tau$-periodic matrix function $A(t)=\nabla g(\gamma(t))$, where $\nabla g$ is the Jacobian matrix. Consider $\Phi(t)$ the {\em principal} fundamental matrix solution of $\dot y=A(t)y=\nabla g(\gamma(t))y$, that is the unique fundamental matrix solution so that $\Phi(0)=I$, and assume that a Floquet normal form $\Phi(t)=Q(t)e^{Rt}$ has been computed. Theorem~\ref{theorem:bundle} shows how the information from the Floquet normal form can directly be used to compute important dynamical properties of $\Gamma$. More explicitly, it is demonstrated that the stability of the periodic orbit $\Gamma$ can be determined by the eigenvalues of $R$ while the stable and unstable tangent bundles of $\Gamma$ can be retrieved from the action of $Q(t)$ (with $t \in [0,\tau]$) on the eigenvectors of $R$. 

A final motivation comes from the fact that computing stable and unstable bundles of periodic orbits is an important step toward computing rigorous parameterization of invariant manifolds of periodic orbits. In fact, one of our future goal consists of combining the ideas of \cite{MR2177465} to rigorously parameterize invariant manifolds of periodic orbits, and then to use that information to solve rigorously, following similar ideas than the ones presented in \cite{BLMM}, a projected boundary value problem whose solutions would correspond to cycle-to-cycle connections and to point-to-cycle connections. Note that the approach of using projected boundary value problems to compute (non rigorously) cycle-to-cycle connections and to point-to-cycle connections has been adopted by several authors (e.g. see \cite{MR2057361}, \cite{MR2454068}, \cite{MR2511084}).

Let us now introduce the ideas behind the rigorous method to compute Floquet normal forms. Rather than jumping immediately into a deep mathematical description of the method, we present the general ideas and we refer to Section~\ref{sec:rig_comp} for a more detailed presentation. 

The first step is to substitute the Floquet normal form $\Phi(t)=Q(t)e^{Rt}$ in the differential equation \eqref{eq:syst-A}. From this, it follows that $(R,Q(t))$ is a solution of the differential equation with periodic coefficients $\dot Q=A(t)Q-QR$. On the converse, if a real constant matrix $R$ and a $2\tau$-periodic matrix function $Q(t)$ solve  
\begin{equation} \label{eq:ivp}
\left\{\begin{array}{l}
\dot Q=A(t)Q-QR\\
Q(0)=I,
\end{array}\right.
\end{equation}
then the matrix function $\Phi(t) := Q(t)e^{Rt}$ is the principal fundamental solution of \eqref{eq:syst-A}. Therefore, the problem of computing fundamental matrix solutions in the form $\Phi(t) = Q(t)e^{Rt}$ reduces to find $(R,Q(t))$ satisfying \eqref{eq:ivp}. The next step is to introduce a nonlinear operator $f$ (see Section~\ref{sec:set_up} for details) whose zeros are in one-to-one correspondence with the solutions of \eqref{eq:ivp}. Letting $x=(R,Q_0,Q_1,Q_2,\dots)$, where the $Q_k$'s are the Fourier coefficients of $Q(t)$, the problem of computing Floquet normal forms $\Phi(t) = Q(t)e^{Rt}$ is then equivalent to find $x$ such that $f(x)=0$. By the a priori knowledge of the smoothness of $Q(t)$, the Fourier coefficients $Q_k$'s decay fast, meaning that the solutions of $f(x)=0$ live in a suitable Banach space $\Omega^s$ of rapidly decaying coefficients. To prove existence, in a constructive way, of solutions of the infinite dimensional nonlinear operator equation $f(x)=0$, we use {\em rigorous numerics}. To be more precise, the goal of rigorous numerics is to construct algorithms that provide an approximate solution to the problem together with precise bounds within which the exact solution is guaranteed to exist in the mathematically rigorous sense. It is worth mentioning that by now, the use of rigorous numerical methods is a standard approach to study differential equations and dynamical systems (e.g. see \cite{MR2534406}, \cite{MR2591840}, \cite{MR2594444}, \cite{MR1838755}, \cite{MR2049869}, \cite{MR1639986}, \cite{MR1276767}, \cite{MR1701385}, \cite{MR2443030}, \cite{MR1870856}).

Based on the previous discussion, the next step consists of computing a numerical approximation $\bx$ of $f(x)=0$ and to demonstrate that close to $\bx$, there exists a genuine solution $x^*$ of $f(x)=0$, corresponding to the wanted explicit Floquet normal form of the principal fundamental matrix solution $\Phi(t)$ of \eqref{eq:syst-A}. However, since the operator $f$ is infinite dimensional, a finite dimensional approximation of $f$ must be introduced in order to compute an approximation $\bx$. This is done in Section~\ref{sec:finite_dim_reduction}. Once $\bx$ is computed, a Newton-like operator $T:\Omega^s \rightarrow \Omega^s$ defined by $T(x)=x-Af(x)$ is introduced, where $A$ is an injective linear operator which acts as an approximation for $Df(\bx)^{-1}$. Since $A$ is injective, the fixed points of $T$ and the zeros of $f$ are in one-to-one correspondence. The next step is to consider small balls $B_{\bx}(r) \subset \Omega^s$ centered at the numerical approximation $\bx$, and to solve for $r$ for which $T:B_{\bx}(r) \rightarrow B_{\bx}(r)$ is a contraction (see Section~\ref{sec:T}). The rigorous verification that $T$ is a contraction on $B_{\bx}(r)$ is done via the use of the so-called {\em radii polynomials} which provide, in the context of differential equations, an efficient means of determining a domain on which the contraction mapping theorem is applicable. The notion of the radii polynomials was originally introduced in \cite{MR2338393} and \cite{MR2487806} to prove existence of equilibria of PDEs. It was later on adapted to prove existence of equilibria of high-dimensional PDEs (e.g. see \cite{MR2718657}, \cite{secondary_bif}, \cite{CH_smooth}, \cite{GL1}), of periodic orbits of delay equations and PDEs (e.g. see \cite{KL}, \cite{MR2592879}, \cite{time_periodic_PDEs}, \cite{MR2630003}) and connecting orbits of ODEs \cite{BLMM}. We refer to \cite{CGL} for a more extensive and general exposure of the radii polynomials.

In this work, we present a general formulation of the radii polynomials adapted to the context of computing rigorously Floquet normal forms 
(see Section~\ref{sec:radii_polynomials} for more details). We present the explicit bounds in Section~\ref{sec:construction} that lead directly to their construction. Note that these bounds ensure that the truncation error terms, inevitably introduced by computing on a finite dimensional projection, are controlled. It is also important to mention that in the computation of the bounds, the floating point errors are controlled by using interval arithmetic \cite{MR0231516}. In fact, all rigorous computations were performed in {\em Matlab} with the interval arithmetic package {\em Intlab} \cite{Ru99a}.

The paper is organized as follows. In Section~\ref{sec:rig_comp}, we introduce the rigorous numerical method to compute Floquet normal forms $\Phi(t)=Q(t)e^{Rt}$ of fundamental matrix solutions of systems of the form \eqref{eq:syst-A}. In Section~\ref{sec:computing_bundles}, we demonstrate how to use the information from Floquet normal forms to compute stable and unstable bundles of periodic orbits of vector field and how to determine the stability properties of periodic orbits. The main result of this section is Theorem~\ref{theorem:bundle}. Finally in Section~\ref{sec:applications}, we present some applications, where we construct rigorously tangent stable and unstable bundles of some periodic orbits of the Lorenz equations and of the $\zeta^3$-model. 

\section{Rigorous computation of Floquet normal forms} \label{sec:rig_comp}

In this section, we introduce the rigorous numerical method to compute Floquet normal forms $\Phi(t)=Q(t)e^{Rt}$ of fundamental matrix solutions of systems of the form \eqref{eq:syst-A}.
%
As already mentioned in Section~\ref{sec:Intro}, the first step is to introduce a nonlinear operator $f$ whose zeros are in one-to-one correspondence with the solutions of \eqref{eq:ivp}.

\subsection{Set-up of the operator equation \boldmath $f(x)=0$ \unboldmath}
\label{sec:set_up}

In the following $Mat(n,\mathbb R), Mat(n,\mathbb C)$ denote the space of $n\times n$ matrices respectively with real and complex entries. The assumption on $Q(t)$ to be real and $2\tau$-periodic allows to consider the expansion
\begin{equation}\label{eq:Q(t)}
Q(t)=Q_{0}+\sum_{k\in\mathbb Z\setminus\{0\}}\left(Q_{k,1}+iQ_{k,2}\right)e^{ik\tfrac{2\pi}{2\tau}t}, 
\end{equation}
where the Fourier coefficients $Q_{0},Q_{k,i}\in Mat(n,\mathbb R)$  satisfy $Q_{-k,1}=Q_{k,1}$ and $Q_{-k,2}=-Q_{k,2}$ for any $k\geq 1$. 
Being $\tau$-periodic, the matrix-valued function $A(t)$ is also $2\tau$-periodic, thus it makes sense to consider the expansion
\begin{equation}\label{eq:A(t)}
A(t)=\sum_{k\in \mathbb Z}\mathcal A_{k}e^{ik\tfrac{2\pi}{2\tau}t},
\end{equation}
where $\mathcal A_{0}\in Mat(n,\mathbb R)$, while the matrices $\mathcal A_{k}\in Mat(n,\mathbb C)$ satisfy $\mathcal A_{-k}=\mathcal C({\mathcal A}_{k})$, for any $k\geq 1$. Here $\mathcal C({\mathcal A})$ stands for the matrix whose entries are the complex conjugates of the entries of $\mathcal A$.  It has to be remarked that the assumption for $A$ to be $\tau$-periodic implies that $\mathcal A_{k}=0$ for $k$ odd and $\mathcal A_{2l}=\hat{\mathcal A}_{l}$ where $\hat{\mathcal A}_{l}$ is the $l$-th Fourier coefficient of $\mathcal A(t)$ in the basis $\{ e^{ik\tfrac{2\pi}{T}t}\}_{k}$.

After substituting the expansions \eqref{eq:Q(t)} and \eqref{eq:A(t)} in problem \eqref{eq:ivp}, the latter system of ODEs moves into an equation $F(t)=0$, where $F(t)$ is a $2\tau$-periodic matrix function.  By a subsequent projection of $F(t)$ in  the Fourier basis $\{ e^{i k \tfrac{2\pi}{2\tau}t}\}$, it follows that solving \eqref{eq:ivp} is equivalent to solve for the unknowns
$$
R,Q_{0} \in Mat(n,\mathbb R) ~~{\rm and} ~~ Q_{k}\bydef \left( Q_{k,1}, Q_{k,2} \right)
\in Mat(n,\mathbb R)^2
$$
 the infinite dimensional algebraic system 
\begin{equation}\label{eq:four}
f(R,Q_0,\dots,Q_k,\dots)=0\\
\end{equation}
$$
f=(f_{\star}, f_{0}, f_{1},\dots , f_{k},\dots)
$$
defined by
\begin{equation}\label{eq:fk}
  \begin{array}{l}
\displaystyle{  f_{\star}\bydef Q_{0}+ 2\sum_{k\geq 1}Q_{k,1} - I  } \\ 
  \\
  f_{0}\bydef Q_{0}R-(A\cdot Q)_0 \\
  \\ 
  f_k\bydef \left[\begin{array}{l}
f_{k,1}\\
\noalign{\vskip 5pt} f_{k,2}
\end{array}  \right]=\left[\begin{array}{l}
-k\frac{2\pi}{2\tau}Q_{k,2}+Q_{k,1} R -(A\cdot Q)_{k,1}\\
\noalign{\vskip 5pt} \ \ \,k\frac{2\pi}{2\tau}Q_{k,1}+Q_{k,2} R -(A\cdot Q)_{k,2}
\end{array}  \right] ,\quad k\geq 1 \end{array}
 \end{equation}
where $(A\cdot Q)_{k,1}, (A\cdot Q)_{k,2}$ denote respectively the real and imaginary part of the convolution  
$$
(A\cdot Q)_{k}\bydef\sum_{k_1+k_2=k}\mathcal A_{k_1} (Q_{k_2,1}+iQ_{k_2,2}) .
$$
Note that $f_{\star}, f_0 \in Mat(n,\mathbb R)$ and $f_k \in Mat(n,\mathbb R)^2$ for every $k \ge 1$.
 

The problem \eqref{eq:four} consists of:  {\em i)} a system of $n^{2}$ real scalar equations for $f_{\star}=0$ representing the initial condition $Q(0)=I$; {\em ii)} $n^{2}$ real scalar equations for system $f_0=0$ that reproduces $<F(t),1>=0$;  {\em iii)} $2n^{2}$ real scalar equations for each $f_k=0$ ($k\geq 1$).  Note that $f_{k,1},f_{k,2}$ are the real and complex part of the equation $<F(t),e^{ik\tfrac{2\pi}{2\tau}t}>=ik\frac{2\pi}{2\tau}Q_k+Q_k R -(\mathcal A\cdot Q)_k$. Here, $<\cdot,\cdot>$ represents the inner product in $L^2\left( [0,\frac{2 \pi}{2 \tau }] \right)$.

Before proceeding with the analysis of the system $f=0$ given by \eqref{eq:four},
 let us introduce some notation that will be adopted throughout the paper.
\vskip 5pt
{\it Notation}

Let  $A,B$ be  matrices with entries $A=\{a_{i,j}\}$, $B=\{b_{i,j}\}$ and $\mathcal{A}=(A_{1},\dots,A_{n}), \\ \mathcal{B}=(B_{1},\dots,B_{n})$ be vectors of matrices. Denote by
\begin{itemize}
\item[i)] $|A|=\{|a_{i,j}|\}$ the matrix of absolute values, where $|\cdot |$ denotes both the real and complex absolute value, according with $a_{i,j}$. For vectors $|\mathcal A|=(|A_{1}|,\dots,|A_{n}|)$; \\
$|A|_{\infty}=\max_{i,j}\{|a_{i,j}|\}$ and $|\mathcal A|_{\infty}=\max\{|A_{1}|_{\infty},\dots,|A_{n}|_{\infty} \}$
\item[ii)] $A\leq_{cw} B$ means $a_{i,j}\leq b_{i,j}$ for any $i,j$. In case $b$ is a scalar, $A\leq_{cw}b$ means $a_{i,j}\leq b$. In case of vectors $\mathcal A\leq_{cw}\mathcal B$ and $\mathcal A\leq_{cw} b$ extends as $A_{k}\leq_{cw} B_{k}$ and $A_{k}\leq_{cw} b$, for any $k=1\dots n$. The same for $\geq_{cw}, >_{cw}, <_{cw}$;
\item[iii)] $\| A\|_{\infty} $ is the standard infinity norm of a matrix: $\|A\|_{\infty}=\max_{i}\sum_{j}|a_{i,j}|$;
\item[iv)] $I$ denotes the identity $n\times n$ matrix, $\mathds 1_{n}$ is the $n\times n$ matrix whose entries are all $1$.
\end{itemize}

Coming back to the analysis of system \eqref{eq:four} let us define the space  
$$
X=\left\{x=(x_{0},x_{1},\dots, x_{k},\dots):  \begin{array}{l}
x_{0}=(R,Q_{0})\in Mat(n,\mathbb R)^{2}\\
\noalign{\vskip 2pt}x_{k}=Q_{k}=(Q_{k,1},Q_{k,2})\in Mat(n,\mathbb R)^{2},k\geq 1
\end{array}\right\}.
$$
Note that 
$f:X\rightarrow X$.
Later on the problem of solving $f=0$ will be transformed into a fixed point problem for an operator $T$:  that requires the choice of a suitable Banach subspace of $X$ where to investigate the existence of solutions.   
To define the proper Banach space, let us first introduce 
   the weigh function
\begin{equation}
w_k=\left\{\begin{array}{ll}
|k|\quad &k\neq 0\\
1  &k=0
\end{array}
\right. 
\end{equation}
 and given $x= (R,Q_{0},Q_{1,1},Q_{1,2},\dots,Q_{k,1},Q_{k,2},\dots) \in X$, let us define the $s$-norm of $x$ in $X$ by
 $$
 \|x\|_{s}\bydef\sup_{k\geq 0}\{|x_{k}|_{\infty}w_{k}^{s} \}=\sup\Big\{ |R|_{\infty},|Q_{0}|_{\infty},\sup_{k\geq 1}\{|Q_{k,1}|_{\infty}w_{k}^{s},|Q_{k,2}|_{\infty}w_{k}^{s}\}\Big\} .
 $$
According with the $s$-norm, let us define the space $\Omega^s$  of sequences in $X$ with algebraically decaying tails 
\begin{equation} \label{eq:Omega_s}
\Omega^s=\{x\in X : \| x\|_{s}<\infty\} .
\end{equation}
For any $s>0$ the space $\Omega^s$ endowed with the $s$-norm is a Banach space and the inclusion $\Omega^s\supset \Omega^{s+1}$ holds. 
The introduction of $\Omega^s$ is motivated by the fact that  a periodic  solution $Q(t)$ of system  \eqref{eq:ivp} results to be at least as smooth as $A(t)$. Thus, in case  the function $A(t)$ is analytic, it follows that $Q(t)$ is analytic. As a consequence  the  Fourier coefficients of $Q(t)$  decay faster than any power rate and therefore they live in $\Omega^s$ for any $s$. On the other hand, even a weaker assumption of the function $A(t)$, such as a $|\mathcal A_{k}|_{\infty}<Cw_{k}^{-s}$ for a constant $C$ and positive $s$, allows to conclude that the solution $x\in \Omega^s$.  The latter is the case we are mainly interested in. Indeed, for the sake of generality and to emphasize  the robustness and versatility of the technique, one assumes  the weakest assumption on $\mathcal A_{k}$ that makes the computational method applicable.   Such assumption is  that there exists $s\geq 2$ and a constant $C>0$ such that the coefficient $\mathcal A_{k}$ satisfy $|\mathcal A_{k}|_{\infty}<Cw_{k}^{-s}$. This condition implies that an integrable function $A(t)$ with expansion as in   \eqref{eq:A(t)} is differentiable up to order $s-1$.



Denote with $\mathcal A=\{\mathcal A_{k}\}_{k\geq 0}$ the sequence of the Fourier coefficients appearing in \eqref{eq:A(t)} and, as an extension of the $s$-norm, define
\begin{equation} \label{eq:calA_norm}
\|\mathcal A\|_{s}=\sup_{k\geq 0}\{|\mathcal A_{k}|_{\infty}w_{k}^{s}\}.
\end{equation}

\begin{lemma}\label{lemma:f}
Assume $\|\mathcal A\|_{s^{\star}}<\infty$  for $s^{\star}\geq 2$. Then $f$ maps $\Omega^s$ in $\Omega^{s-1}$, for any $2\leq s\leq s^{\star}$.
\end{lemma}
\begin{proof} Let $2\leq s\leq s^{\star}$ and suppose $x\in \Omega^s$ . Then  $|\mathcal A_{k}|_{\infty}<C_{1}w_{k}^{-s}$ and, from Lemma 2.1 in \cite{MR2718657},  $|(\mathcal A \cdot Q)_{k}|_{\infty}\leq \frac{C_{2}}{w_{k}^{s}}$ . Thus $|f_{k}(x)|_{\infty}\leq C_{3}k|Q_{k}|_{\infty}+C_{4}|Q_{k}|_{\infty}+C_{2}w_{k}^{-s}<Cw_{k}^{-s+1}$, for suitable constants $C, C_{i}$. This shows that $f(x)\in\Omega^{s-1}$.
\end{proof}

Thus we will look for solutions of the system \eqref{eq:four} within the space $\Omega^s$ for some $s\geq2$. The idea is to reformulate the zero finding problem $f(x)=0$ as a fixed point problem for a suitable operator $T$ defined in $\Omega^s$ and to verify the hypothesis of the contraction mapping theorem in order to conclude about the existence of a fixed point. More explicitly, the idea is to prove the existence of a ball $B_{\bar x}(r)$ in $\Omega^{s}$ around a finite dimensional approximate solution $\bar x$ on which the operator $T$ is a contraction. The proof will follow by verifying a finite number of polynomial inequalities: the so-called {\em radii polynomials}. Their computation will result from rigorous numerical computations and analytic estimates. The next step is to compute a finite dimensional approximate solution $\bar x$. For this, one needs to introduce a finite dimensional projection of $f(x)=0$ given by \eqref{eq:four}.

\subsection{Finite dimensional projection} \label{sec:finite_dim_reduction}
As mentioned earlier, the fist step involved in the computational method is to consider a finite dimensional projection 
and to compute an approximate numerical solution of \eqref{eq:four}. 

For $m>1$ consider the finite dimensional space $X^{m}=\prod_{k=1}^{m}Mat(n,\mathbb R)^{2}$ and define the projections 
$$
\begin{array}{rl}
\Pi_{m}:X&\rightarrow X^{m}\\
x&\mapsto \Pi_{m}(x)=x^{m}=(R,Q_{0},\dots,Q_{m-1})\\
\\
\Pi_{\infty}:x&\mapsto (Q_{m},Q_{m+1},\dots)
\end{array}
$$
so that $x=(x^{m},\Pi_{\infty}(x))$. Denote with $0^{\infty}\bydef\Pi_{\infty}(0)$. Moreover let us define the restricted map
\begin{equation} \label{eq:finite_dim_reduction}
\begin{array}{rl}
f^{(m)}: X^{m}&\rightarrow X^{m}\\
x^{m}&\mapsto \Pi_{m}f(x^{m},0^{\infty})
\end{array}
\end{equation}
Note that for any $x\in X$ the sequence $(x^{m},0^{\infty})\in X$ and the finite dimensional projection $\Pi_{m}$ applied to $f(x)$ reads as $\Pi_{m}f(x)=(f_{\star},f_{0},\dots,f_{m-1})(x)$. Since $X^m$ is isomorphic to $\mathbb{R}^{m2n^2}$, one can think of $f^{(m)}:\mathbb{R}^{m2n^2} \rightarrow \mathbb{R}^{m2n^2}$.
Suppose that using a Newton-like iterative algorithm, one computed   
$$
\bar x=(\bar R,\overline{Q}_{0},\dots,\overline Q_{m-1})
$$
 an approximate zero of $f^{(m)}$, that is $f^{(m)}(\bar x)\approx 0$. For simplicity the same notation $\bar x$ is used to identify the above vector in $X^{m}$ and the sequence $(\bar x,0^{\infty})$ in $X$.  As already mentioned at the end of Section~\ref{sec:set_up}, the idea is to consider a ball $B_{\bar{x}}(r) \in \Omega^s$ centered at the approximate solution $\bar x$ and to show the existence of a contraction mapping $T$ acting on $B_{\bar{x}}(r)$. Hence, let us now introduce the fixed point operator $T$.

\subsection{The fixed point operator \boldmath $T(x)=x$ \unboldmath} \label{sec:T}

In this section, we first define an operator $T$ on $\Omega^{s}$ whose fixed points correspond to solutions of $f(x)=0$ and then, we introduce some computable conditions from which one can conclude about the existence of fixed point of $T$. To begin with, suppose to have chosen a representation of the matrices $Mat(n,\mathbb R)$ as vector in $\mathbb R^{n^{2}}$ and to have extended it to an isomorphism between the space of sequences of $N$ matrices $Mat(n,\mathbb R)$ to $\mathbb R^{Nn^{2}}$.  Note that $X^{m}$ is isomorphic to $\mathbb R^{m2n^{2}}$.

In the sequel, consider a vector $V = [v_{1},\dots,v_{N2n^{2}}] \in \mathbb R^{N2n^{2}}$. We denote by $V_{k} \in \mathbb{R}^{2n^2}$, $k=0,\dots, N-1$ the vector with $2n^{2}$ components $V_{k}=[v_{k2n^{2}+1},v_{k2n^{2}+2},\dots,v_{(k+1)2n^{2}}]$. The reason of this choice of notation is the following: suppose that $V$ is the vector representation of the sequence  $x=(R,Q_{0},Q_{1},\dots,Q_{N-1})\in X^{N}$ for a positive $N$, then $V_{k}$ represents the couple $(R,Q_{0})$ when $k=0$ and $Q_{k}=(Q_{k,1},Q_{k,2})$ for $k\geq 1$.  

 
 Denote by $Df^{(m)}(\bar x)$ the Jacobian of $f^{(m)}$ with respect to $x^m$ evaluated at $\bar x$, that is 
$$
Df^{(m)}\bydef Df^{(m)}(\bar x)=\frac{\partial( f_{\star},f_0,f_1,\dots, f_{m-1})}{\partial(R,Q_0,\dots, Q_{m-1})}(\bar x) \in Mat(2n^{2}m,\mathbb R) .
$$
For clarity and completeness, 
\begin{equation}
Df^{(m)}=\left[\begin{array}{cccc}
\begin{array}{l}
\frac{\partial f_{\star}}{\partial R} \quad \frac{\partial f_{\star}}{\partial Q_{0}}\\
\noalign{\vskip 2pt}\frac{\partial f_0}{\partial R} \quad \frac{\partial f_0}{\partial Q_{0}}
\end{array}
&
\begin{array}{l}
\frac{\partial f_{\star}}{\partial Q_{1,1} }\quad  \frac{\partial f_{\star}}{\partial Q_{1,2}}\\
\noalign{\vskip 2pt}\frac{\partial f_{0}}{\partial Q_{1,1} }\quad  \frac{\partial f_{0}}{\partial Q_{1,2}}
\end{array}
&
\begin{array}{l}
\dots
\end{array}
& 
\begin{array}{l}
 \frac{\partial f_{\star}}{\partial Q_{m-1,1} } \quad \frac{\partial f_{\star}}{\partial Q_{m-1,2}}\\
\noalign{\vskip 2pt}   \frac{\partial f_{0}}{\partial Q_{m-1,1} } \quad \frac{\partial f_{0}}{\partial Q_{m-1,2}}
\end{array}
\\
\noalign{\vskip 2pt}  \frac{\partial f_{1}}{\partial (R,Q_{0}) }&\frac{\partial f_{1}}{\partial Q_{1} }&\dots&\frac{\partial f_{1}}{\partial Q_{m-1} }\\
\noalign{\vskip 2pt}  \vdots & \vdots & \vdots & \vdots  \\
\noalign{\vskip 2pt}  \frac{\partial f_{m-1}}{\partial (R,Q_{0})  }&\frac{\partial f_{m-1}}{\partial Q_{1} }&\dots&\frac{\partial f_{m-1}}{\partial Q_{m-1} }
\end{array}
\right](\bar x)
\end{equation}
where for $k,j=1,\dots , m-1$
$$
 \frac{\partial f_{k}}{\partial (R,Q_{0})}=\left[\begin{array}{l}
  \frac{\partial f_{k,1}}{\partial R }\quad  \frac{\partial f_{k,1}}{\partial Q_{0} }\\
  \noalign{\vskip 2pt}  \frac{\partial f_{k,2}}{\partial R } \quad  \frac{\partial f_{k,2}}{\partial Q_{0} }
 \end{array}\right],\quad \frac{\partial f_{k}}{\partial Q_{j} }=\left[\begin{array}{l}
  \frac{\partial f_{k,1}}{\partial Q_{j,1} }\quad  \frac{\partial f_{k,1}}{\partial Q_{j,2} }\\
 \noalign{\vskip 2pt}   \frac{\partial f_{k,2}}{\partial Q_{j,1} } \quad  \frac{\partial f_{k,2}}{\partial Q_{j,2} }
 \end{array}\right],
$$
and each $ \frac{\partial f_{k,i}}{\partial Q_{j,l} }\in Mat(n^{2},\mathbb R)$ denotes the Jacobian matrix of the components of $f_{k,j}$ with respect to the components of $Q_{j,l}$.  

Suppose to have numerically computed $Df^{(m)}$ and denote by $A_m \in Mat(2n^{2}m,\mathbb R)$ an invertible numerical approximation of  $(Df^{(m)})^{-1}$
$$
 A_m\cdot Df^{(m)}\approx I
 $$
  and for $k\geq m$, define
$$
\Lambda_k\bydef \frac{\partial f_k}{\partial Q_k}(\bar x) \in Mat(2n^{2},\mathbb R).
$$

\begin{lemma}\label{lemma:Lambda_invertible}
Recall \eqref{eq:A(t)} and \eqref{eq:calA_norm}, and assume that $\|\mathcal A\|_{s}<\infty$ for some $s\geq 2$. Then there exist two constants $K$ and $C_\Lambda$  such that for any  $k\geq K$  the linear operator $\Lambda_{k}$ is invertible and $\| \Lambda_{k}^{-1}\|_{\infty}<\frac{C_\Lambda}{k}$. The constants $K$ and $C_\Lambda$ depend on $\|\mathcal A\|_{s}$, the period $\tau$ and $|\bar R|_{\infty}$.
\end{lemma}
\begin{proof}
The real and imaginary parts of $(A\cdot Q)_{k}$ can be written explicitly as
$$
(A\cdot Q)_{k,1}=(Re(\mathcal A_{0})+Re(\mathcal A_{2k}))Q_{k,1}+Im(\mathcal A_{2k})Q_{k,2}+W_{1}
$$
$$
(A\cdot Q)_{k,2}=Im(\mathcal A_{2k})Q_{k,1}+(Re(\mathcal A_{0})-Re(\mathcal A_{2k}))Q_{k,2} +W_{2}
 $$
 where $W_1$ and $W_2$ do not depend on $Q_{k,1}$ and $Q_{k,2}$.
 Thus, looking at the definition of $f_{k}$ in \eqref{eq:fk},  it follows that $\Lambda_{k}$ is of the form
\begin{equation} \label{eq:Lamnda_k}
\Lambda_{k}=\left[ \begin{array}{cc}
\lambda_{1,1} \quad & -k\tfrac{2\pi}{2\tau} \mathbb I_{n^{2}}+\lambda_{1,2}\\
k\tfrac{2\pi}{2\tau} \mathbb I_{n^{2}}+\lambda_{2,1} & \lambda_{2,2} 
\end{array}\right],
\end{equation}
where the entries of $\lambda_{1,1}$ and $\lambda_{2,2}$ are linear combination of the entries of $\bar R$, $\mathcal A_{0}, \mathcal A_{2k}$ so that
$|\lambda_{1,1}|_{\infty},|\lambda_{2,2}|_{\infty}<|\bar R|_{\infty}+|\mathcal A_{0}|_{\infty}+|\mathcal A_{2k}|_{\infty}$ holds. 
Also, $\lambda_{2,1}=\lambda_{1,2}$ only depend on $\mathcal A_{2k}$. 
By a row permutation, the invertibility of $\Lambda_{k}$ is equivalent to the invertibility of
$$
\hat\Lambda_{k}=\left[ \begin{array}{cc}
k\tfrac{2\pi}{2\tau} \mathbb I_{n^{2}}+\lambda_{2,1} & \lambda_{2,2} \\
\lambda_{1,1} \quad & -k\tfrac{2\pi}{2\tau} \mathbb I_{n^{2}}+\lambda_{1,2}\\
 \end{array}\right].
$$ 
Since $|\lambda_{1,1}|_{\infty},|\lambda_{2,2}|_{\infty}<|\bar R|_{\infty}+|\mathcal A_{0}|_{\infty}+|\mathcal A_{2k}|_{\infty}$ and $|\lambda_{1,2}|_{\infty}<|\mathcal A_{2k}|_{\infty}$, the assumption $\| \mathcal A\|_{s}<\infty$ implies that the $|\lambda_{i,j}|_{\infty}$ are uniformly bounded in $k$, and moreover $|\lambda_{1,2}|_{\infty}$ is  decreasing. Thus there exists $K$ such that $\hat\Lambda_{k}$ is diagonally dominant for any $k\geq K$. This is enough to conclude that $\hat \Lambda_{k}$ is invertible for any $k\geq K$. 

Denote by $a_{i,i}$ the diagonal elements of $\hat\Lambda_{k}$. Hence, if $\hat \Lambda_{k}$ is diagonally dominant, that is if $|a_{i,i}|>\sum_{j\neq i}|\hat \Lambda_{k}(i,j)|$ for any $i=1,\dots, 2n^{2}$, then using a result from \cite{varga}, one gets the following bound 
$$
\| \hat\Lambda_{k}^{-1}\|_{\infty}\leq\max_{i}\left\{\frac{1}{|a_{i,i}|-\sum_{j\neq i}|\hat \Lambda_{k}(i,j)| }   \right\}.
$$
Therefore, for $k\geq K$
$$
\| \Lambda_{k}^{-1}\|_{\infty}=\|\hat\Lambda_{k}^{-1}\|_{\infty}\leq \frac{C_\Lambda}{k}
$$
for a constant $C_\Lambda$ depending on $\tau$, $\bar R$, $|\mathcal A_{0}|_{\infty}$ and $|\mathcal A_{2k}|_{\infty}$.
\end{proof}

Suppose that we chose the finite dimensional parameter $m>K$ where $K$, as defined in Lemma~\ref{lemma:Lambda_invertible}, is such that $\Lambda_{k}$ is invertible for any $k\geq K$. 
A formal diagonal concatenation of the operator $A_m$ and the sequence $\Lambda^{-1}_{k}$, for $k\geq m$, produces the linear operator 
 \begin{equation}
 \begin{array}{c}
 A:X\rightarrow X\\
 \qquad x\mapsto Ax\\
 (Ax)_{k}\bydef\left\{\begin{array}{cl}
 \Big(A_mx^{m}\Big)_{k} & k=0,\dots,m-1\\
 \Lambda^{-1}_{k}Q_{k} & k\geq m.
 \end{array}\right.
 \end{array}
 \end{equation}
We define the operator $T$ on $X$ as
$$
T(x)\bydef x-A f(x),
$$
and denote $T_{k}(x)=(T(x))_{k}$.
\begin{lemma}\label{lemma:f=0}
Recall \eqref{eq:A(t)} and \eqref{eq:calA_norm}, and assume that $\|\mathcal A\|_{s^{\star}}<\infty$ for $s^{\star}\geq 2$. Then  for any $2\leq s\leq s^{\star}$, $T:\Omega^s\rightarrow \Omega^s$ and solutions of $T(x)=x$ correspond to solutions of $f(x)=0$. 
\end{lemma}
\begin{proof}
From Lemma \ref{lemma:f}, given $x\in \Omega^s$ it follows that $f(x)\in \Omega^{s-1}$. The linear operator $A$ maps $\Omega^{s-1}$ in $\Omega^s$. Indeed for $k\geq m$, $|(Ax)_{k}|_{\infty}=|\Lambda^{-1}_{k}x_{k}|_{\infty}\leq \|\Lambda^{-1}_{k}\|_{\infty}|x_{k}|_{\infty}$.  Thus from Lemma \ref{lemma:Lambda_invertible} and assuming $x\in \Omega^{s-1}$, it follows that $|(Ax)_{k}|_{\infty}\leq\frac{C}{k}\frac{\|x\|_{s-1}}{w_{k}^{s-1}}<\frac{C_{1}}{w_{k}^{s}}$, for positive constants $C,C1$.  This proves that $T:\Omega^s\rightarrow \Omega^s$. Since $A_m$ is invertible by assumption and $\Lambda_{k}$ have been proved in Lemma \ref{lemma:Lambda_invertible} to be invertible for all $k \ge m>K$, it follows that the linear operator $A$ is invertible and therefore fixed points of $T$ correspond to zeros of $f(x)$. 
\end{proof}

By construction, when restricted to the finite dimensional reduction  $\Pi_{m}\Omega^s$, the operator $T$ acts as  $T(x^{m})=x^{m}-A_mf^{(m)}(x^{m})$. Thus on $\Pi_{m}\Omega^s$, $T$ is close  to the  Newton operator: the only difference is that point where the  derivative is computed does not change along the iteration  process. Therefore we can consider $T$ as an extension to a infinite dimensional space of a finite dimensional Newton-like operator. 

The existence of a fixed point for the operator $T$ will be assured by the Banach Fixed Point Theorem once the operator  $T$ has been proved to be a contraction on a suitable ball in $\Omega^s$. The suitable ball on which $T$ will be proved to be a contraction will be sought within the family of balls  $B_{\bar x}(r)\in\Omega^s$
$$
B_{\bar x}(r)=\bar x +B(r)
$$   
where $B(r)$ is the ball of radius $r$ in $\Omega^s$ centered in the origin and $r$ is treated as variable. Following the same approach as in different other papers (e.g. see \cite{MR2718657}, \cite{MR2443030}, \cite{KL}, \cite{time_periodic_PDEs}, \cite{MR2630003}, \cite{MR2592879}, \cite{CH_smooth}, \cite{GL1}, \cite{BLMM}, \cite{MR2338393}, \cite{MR2487806}), we are going to construct a finite  set of computable conditions, the so-called {\em radii polynomials}, to be solved in $r$, whose verification implies that the hypothesis of the Banach Fixed Point Theorem are satisfied. In practice, the radii polynomials are defined as realization of the hypotheses of the following theorem.

Suppose there exist two matrices sequences 
$$Y=(Y_{0},Y_{1},\dots Y_{k},\dots),\quad Z(r)=(Z_{0},Z_{1},\dots Z_{k},\dots)(r),\qquad  Y,Z\in X$$
such that 
\begin{equation}\label{eq:bounds}
|(T(\bar x)-\bar x)_k|\leq_{cw} Y_{k},\quad\sup_{b_1,b_2\in B(r)}\Big|\big[DT(\bar x+b_1)b_2\big]_{k}\Big|\leq_{cw} Z_{k}(r),\quad \forall k\geq 0.
\end{equation}
\begin{theorem}
\label{teor:exist}
Fix $s\geq 2$ and let $Y$ and $Z$ defined as in \eqref{eq:bounds}.
If there exists $r>0$ such that $\parallel Y+Z\parallel_s<r$, then the operator $T$ maps $B_{\bar x}(r)$ into itself and $T:B_{\bar x}(r)\rightarrow B_{\bar x}(r)$ is a contraction. Thus, by the Banach Fixed Point Theorem, there exists an unique $x^* \in B_{\bar x}(r)$ solution of $T(x^*)=x^*$ and therefore solution of $f(x^*)=0$. 
\end{theorem}
\begin{proof}
Two statements need to be proved:
\begin{itemize}
\item[i)] $T(B_{\bar x}(r))\subset B_{\bar x}(r)$, that is $\|T(x)-\bar x\|_{s}<r$ for all  $x\in B_{\bar x}(r)$,
\item[ii)] $T$ is a contraction, that is there exists $\kappa \in (0,1)$ such that for every $x,y\in B_{\bar x}(r)$, one has that $\|T(x)-T(y)\|_{s} \le \kappa \|x-y\|_{s}$.
\end{itemize} 
For a given $k\geq 0$ and any $x,y\in B_{\bar x }(r)$, the mean value theorem implies 
$$
T_{k}(x)-T_{k}(y)=DT_{k}(z)(x-y)
$$
for some $z\in \{tx+(1-t)y\, : t\in[0,1]\}\subset B_{\bar x}(r)$. Note that $r\frac{(x-y)}{\|x-y\|_{s}}\in B(r)$ thus for \eqref{eq:bounds}
\begin{equation}\label{eq:bnd}
|T_{k}(x)-T_{k}(y)|=\left|DT_{k}(z)\frac{r(x-y)}{\|x-y\|_{s}}\right|\frac{1}{r}\|x-y\|_{s}\leq_{cw}\frac{Z_{k}(r)}{r}\|x-y\|_{s}
\end{equation}
The triangular inequality applied component-wise gives
$$
|T_{k}(x)-\bar x_{k}|\leq_{cw}|T_{k}(x)-T_{k}(\bx)|+|T_{k}(\bx)-\bx_{k}|\leq_{cw}Y_{k}+Z_{k}(r)
$$
hence
$$
|T_{k}(x)-\bar x_{k}|_{\infty}\leq|Y_{k}+Z_{k}(r)|_{\infty} .
$$
Therefore for any $x\in B_{\bar x}(r)$
$$
\|T(x)-\bar x\|_{s}=\sup_{k\geq 0}\{ |T_{k}(x)-\bar x_{k}|_{\infty}w_{k}^{s}\}\leq\sup_{k\geq 0}\{|Y_{k}+Z_{k}(r)|_{\infty}w_{k}^{s}\}=\|Y+Z(r)\|_{s}<r .
$$
This proves $i)$.

Again from \eqref{eq:bnd}, for any $x,y\in B_{\bar x}(r)$,  $\displaystyle{|T_{k}(x)-T_{k}(y)|_{\infty}\leq \frac{|Z_{k}(r)|_{\infty}}{r} \|x-y\|_{s}}$, thus
\begin{equation}\label{eq:bnd2}
\|T(x)-T(y)\|_{s}\leq\frac{\|Z(r)\|_{s}}{r}\|x-y\|_{s}
\end{equation}
Note that all the entries of  $Y_{k}$ and $Z_{k}(r)$ are non negative, thus $|Z_{k}(r)|_{\infty}\leq|Y_{k}+Z_{k}(r)|_{\infty}$ and $\|Z(r)\|_{s}\leq\|Y+Z(r)\|_{s}<r$. 
That implies that 
$$
\kappa \bydef \frac{\|Z(r)\|_s}{r} \in (0,1),
$$
and we can conclude the proof of $ii)$. An application of the Banach Fixed Point Theorem on the Banach space $B_{\bar x}(r)$ gives the existence and unicity of a solution $x^*$ of the equation $T(x)=x$ in $B_{\bar x}(r)$ and, from Lemma \ref{lemma:f=0}, of a solution of $f(x)=0$.
\end{proof}

\subsection{The radii polynomials} \label{sec:radii_polynomials}

As already mentioned in Section~\ref{sec:Intro}, the radii polynomials are a set of $r$-dependent polynomials $p_{k}(r)$ defined in such a way that if $r^{*}$ is a common solution of $p_{k}(r^{*})<0$, then the ball 
$B_{\bx}(r^*) \subset \Omega^s$ of radius $r^*$ centered at the numerical approximation $r^*$ contains a unique solution of $f(x)=0$.
This is due to the fact that by construction of the polynomials, one has that $\|Y+Z(r^{*})\|_{s}<r^{*}$, meaning that the hypotheses of Theorem~\ref{teor:exist} are satisfied.
In terms of the components, the formula $\parallel Y+Z(r)\parallel_s<r$  reads as
\begin{equation}\label{eq:totsystem}
|Y_k+Z_k(r)|_{\infty}-\frac{r}{w_k^s}<0 , \quad \forall k\geq 0 .
\end{equation}
The latter consists of a system of infinitely many inequalities, which is then impossible to be verify directly with computations. In order to reduce \eqref{eq:totsystem} to a finite number of inequalities, suppose that, for a given $M$, there exist $Y_{ M}$ and $Z_{M}(r)$ such that 
\begin{equation}\label{eq:coda}
|(T(\bar x)-\bar x)_k|_{\infty}\leq \frac{ M^{s}}{k^{s}}Y_{ M},\quad\sup_{b_1,b_2\in B(r)}\Big|\big[DT(\bar x+b_1)b_2\big]_{k}\Big|_{\infty}\leq\frac{M^{s}}{k^{s}} Z_{M}(r),\quad \forall k\geq M,
\end{equation}
and introduce the  set of $ M +1$  radii polynomials as follows.
\begin{definition} \label{def:radii_polynomials}
The {\em radii polynomials} are defined as
\begin{equation}
\begin{array}{l}
p_{k}(r)\bydef Y_{k}+Z_{k}(r)-\frac{r}{w_{k}^{s}} (\mathds{1}_{n},\mathds{1}_{n}),\quad k=0,\dots, M-1\\
p_{ M}\bydef Y_{ M}+ Z_{ M}-\frac{r}{w_{M}^s} .
\end{array}
\end{equation}
\end{definition}

\begin{theorem}
Consider $ M$ and let $Y, Z$ such that  $Y_{k},Z_{k}$ satisfy \eqref{eq:bounds} for $k=0,\dots, M-1$ while for $k\geq  M$ define
$$
Y_{k}\bydef \frac{ M^{s}}{k^{s}}Y_{M}[\mathds 1_{n},\mathds 1_{n}],\quad Z_{k}(r)\bydef \frac{ M^{s}}{k^{s}}Z_{ M}[\mathds 1_{n},\mathds 1_{n}],
$$
where $Y_{M},Z_{ M}$ 
 satisfy the tail condition \eqref{eq:coda}. If there exists $r>0$ such that $p_{k}(r)<_{cw} 0$ for all $k=0\dots, M $, then there exists a unique $x^* \in B_{\bar x}(r)$ such that $T(x^*)=x^*$ and $f(x^*)=0$.
\end{theorem}
\begin{proof}
Since by definition $Y_{k}\geq_{cw}0,Z_{k}\geq_{cw}0$,  the relations $p_{k}(r)<_{cw}0$ imply that $|Y_{k}+Z_{k}(r)|_{\infty}<\frac{r}{w_{k}^{s}}$ for $k=0,\dots,  M-1$. For  $k\geq M$, $Y_{k}, Z_{k}$ satisfy   \eqref{eq:bounds} and from  $Y_{ M}+ Z_{ M}(r)-\frac{r}{w_{ M}^s} <0$,  it follows that $|Y_k+Z_k(r)|_{\infty}-\frac{r}{w_k^s}<0$. Indeed
$$
|Y_k+Z_k(r)|_{\infty}=\frac{ M^{s}}{k^{s}}(Y_{M}+Z_{ M})<\frac{ M^{s}}{k^{s}}\frac{r}{M^{s}},\quad\forall k \ge M.
$$
Hence
$$
\|Y+Z\|_{s}=\sup_{k\geq 0}\{|Y_{k}+Z_{k}|w_{k}^{s}\}<r,
$$
and the result follows from Theorem~\ref{teor:exist}.
\end{proof}

\noindent

\subsection{Construction of the bounds \boldmath $Y,Z$  \unboldmath}\label{sec:construction}

This section is devoted to the construction of  the matrices $Y_{k}$, $Z_{k}$ satisfying \eqref{eq:bounds}, and of the asymptotic bounds $Y_{M}$, $Z_{ M}$ satisfying \eqref{eq:coda}. This construction provides the complete description of the radii polynomials introduced in Definition~\ref{def:radii_polynomials}. With the aim of remaining as general as possible, the only constraint we assume on the $\tau$-periodic function $A(t)$ is that the vector of Fourier coefficients $\mathcal A$ given in \eqref{eq:A(t)} satisfies $\|\mathcal A\|_{s^{\star}}<\infty$ for $s^{\star}\geq 2$. Nevertheless, further information on the coefficients $\mathcal A_{k}$ may be useful to get sharper analytical estimates.

In what follows, the growth rate parameter $s$ has been fixed so that $2\leq s\leq s^{\star}$, the finite dimensional parameter $m$ has been chosen greater than $K$, where $K$ is a lower bound given by Lemma \ref{lemma:Lambda_invertible} and the computational parameter $M$ has been chosen so that $M>m$. Moreover, assume that one computed $\Lambda_{k}^{-1}$ for $k=m,\dots,M-1$. Note that in some cases, it will be possible to achieve this task analytically, but in other cases, only an interval enclosure using rigorous numerics will be possible. Also, recalling Lemma \ref{lemma:Lambda_invertible}, denote by $C_{\Lambda}$ a computable constant such that 
\begin{equation} \label{eq:asymp_bound}
\|\Lambda_{k}^{-1}\|_{\infty}\leq\frac{C_{\Lambda}}{k}, ~ {\rm for } ~ k \ge m .
\end{equation}
\subsubsection{The bound \boldmath $Y$  \unboldmath}
By definition, $T(\bar x)-\bar x=-Af(\bar x)$, thus define $Y$ as
\begin{equation}
Y_{k}=\left\{
\begin{array}{ll}
\left|\left(A_mf^{(m)}(\bar x)\right)_{k}\right|,\quad& k=0\dots,m-1\\
\left|\Lambda_{k}^{-1}f_{k}(\bar x)\right|,& k=m,\dots, M-1.
\end{array}\right.
\end{equation}
The tail bound $Y_{ M}$ is defined so that $Y_{ M}\frac{M^{s}}{k^{s}}>\left|\Lambda_{k}^{-1}f_{k}(\bar x)\right|_{\infty}$, for any $k> M$. We now introduce a coarse bound $Y_{ M}$ based on the relation $\left|\Lambda_{k}^{-1}f_{k}(\bar x)\right|_{\infty}\leq \|\Lambda^{-1}_{k}\|_{\infty}|f_{k}(\bar x)|_{\infty}$. Since  $\bar Q_{k,1}=\bar Q_{k,2}=0$  for any $k \ge m$, it follows that
\begin{equation}
f_{k}(\bar x)=\left[\begin{array}{l}
-(A\cdot Q)_{k,1}\\
\noalign{\vskip 4pt}-(A\cdot Q)_{k,2}\\
\end{array} \right]=\displaystyle{\sum_{\substack{
k_1+k_2=k \\
|k_2|< m
}}}\left[\begin{array}{l}
-Re\left(\mathcal A_{k_1}(\bar Q_{k_2,1}+i\bar Q_{k_2,2})\right)\\
\noalign{\vskip 4pt}-Im\left(\mathcal A_{k_1}(\bar Q_{k_2,1}+i\bar Q_{k_2,2})\right)
\end{array} \right],\quad \forall k\geq  M.
\end{equation}
Now, using the fact that $|\mathcal A_{k}|_{\infty}\leq\|\mathcal A\|_{s^{\star}}w_{k}^{-s^{\star}}$, both $|f_{k,1}(\bar x)|$ and $|f_{k,2}(\bar x)|$ are component-wise  bounded by
\begin{equation*}\begin{array}{ll}
\left|\displaystyle{\sum_{\substack{
k_1+k_2=k \\
|k_2|< m
}}}\mathcal A_{k_1}(\bar Q_{k_2,1}+i\bar Q_{k_2,2})\right|&\leq_{cw}\displaystyle{\sum_{\substack{
k_1+k_2=k \\
|k_2|< m}}}|\mathcal A_{k_1}||\bar Q_{k_2,1}+i\bar Q_{k_2,2}|\\
&\leq_{cw} |\mathcal A_{k}||\bar Q_{0}|+\displaystyle{\sum_{l=1}^{m-1}}(|\mathcal A_{k-l}|+|\mathcal A_{k+l}|)|\bar Q_{l,1}+i\bar Q_{l,2}| \\
&\leq_{cw}\frac{\|\mathcal A\|_{s^{\star}}}{w_{k}^{s}}\left[\frac{w_{k}^{s}}{w_{k}^{s^{\star}}} \mathds 1_{n}|\bar Q_{0}|+\displaystyle{\sum_{l=1}^{m-1}}w_{k}^{s}\left(\frac{1}{w_{k+l}^{s^{\star}}}+\frac{1}{w_{k-l}^{s^{\star}}} \right)\mathds 1_{n}|\bar Q_{l,1}+i\bar Q_{l,2}|\right] .
\end{array}
\end{equation*}
For $k\geq M$ the bounds $\frac{w_{k}^{s}}{w_{k}^{s^{\star}}}\leq 1$ and $w_{k}^{s}\left(\frac{1}{w_{k+l}^{s^{\star}}}+\frac{1}{w_{k-l}^{s^{\star}}}\right)\leq 1+\left(1-\frac{l}{ M} \right)^{-s}$ hold, thus one computes the matrix
$$
W=\mathds 1_{n}|\bar Q_{0}|+\displaystyle{\sum_{l=1}^{m-1}}\left(1+\left(1-\frac{l}{M} \right)^{-s}\right)\mathds 1_{n}|\bar Q_{l,1}+i\bar Q_{l,2}|
$$
so that
$$
 |f_{k}(\bar x)|_{\infty}\leq k^{-s}\|\mathcal A\|_{s^{\star}}|W|_{\infty},\quad {\rm for } ~ k\geq  M.
$$
Finally, using $\|\Lambda_{k}^{-1}\|_{\infty}\leq\frac{C_{\Lambda}}{ M} $, define
$$
Y_{ M}\bydef\frac{1}{ M^{s+1}}\|\mathcal A\|_{s^{\star}}C_{\Lambda}|W|_{\infty} .
$$
\subsubsection{The bound \boldmath $Z$  \unboldmath}

To construct the bound $Z$ so that
$$
\sup_{b_1,b_2\in B(r)}\Big|\big[DT(\bar x+b_{1})b_{2}\big]_{k}\Big|\leq_{cw} Z_{k}(r),\quad \forall k\geq 0,
$$
it is convenient  to factor the points $b_{1},b_{2}\in B(r)$ as $b_{1}=ru$, $b_{2}=rv$ with $u,v\in B(1)$, to expand in the variable $r$ and finally to uniformly  bound the expression using the fact that $u,v\in B(1)$.  Denote  $u=[u_{0},u_{1},\dots,u_{k},\dots]$, where each $u_{k}=(u_{k,1},u_{k,2})\in Mat(n,\mathbb R)^{2}$. In order to simplify the exposition, both the matrices $u_{k,1},u_{k,2}$ will be denoted as $u_{k}$. Indeed, what really matters is the bound $|u_{k,1}|, |u_{k,2}| \leq_{cw}w_{k}^{-s}$ that finally will be applied to obtain the uniform estimates.   The similar notation for $v_{k}$.

Let us introduce the linear operator $ A^{\dag}:\Omega^{s+1}\rightarrow \Omega^s$ defined as 
\begin{equation}
 \left(A^{\dag}x\right)_{k}\bydef\left\{\begin{array}{ll}
\Big(Df^{(m)} \cdot  x^{m}\Big)_{k}\quad &k=0,\dots,m-1\\
\Lambda_kx_k, & k\geq m,
\end{array}\right.
\end{equation} 
and consider the  splitting  
\begin{equation}\label{eq:split}
\begin{array}{rl}
DT(\bar x+ru)rv=&\left[I-ADf(\bar x+ru) \right]rv\\
=&\left[I-A A^{\dag}\right]rv - A\left[Df(\bar x+ru)- A^{\dag}\right]rv.
\end{array}
\end{equation}
The definition of $Z(r)$ will follow as a result of different intermediate estimates: indeed we are going to introduce the vectors $Z^{0}$, $Z^{1}$, $Z^{2}$ such that 
$$
\left|\left[I-AA^{\dag} \right]rv\right|\leq_{cw}Z^{0}r\ , \quad \forall v\in B(1),
$$
$$�
\left|\left[Df(\bar x+ru)- A^{\dag}\right]rv\right|\leq_{cw} Z^{1}r+Z^{2}r^{2},\quad \forall u,v\in B(1).
$$
From \eqref{eq:split} it follows that
\begin{equation}\label{eq:splitnorm}
\Big|\left[DT(\bar x+ru)rv\right]_{k}\Big|\leq_{cw}\Big|\left[\left(I-A A^{\dag}\right)rv\right]_{k}\Big| +\Big|\left[ A\left(Df(\bar x+ru)- A^{\dag}\right)rv\right]_{k}\Big|.
\end{equation}
Hence, the elements $Z_{k}$, for $k=0,\dots M-1$ can be defined as
$$
Z_{k}(r)=\left\{\begin{array}{ll}
Z^{0}_{k}r+\Big[|A_m|(Z^{1}r+Z^{2}r^{2})^{m}\Big]_{k},\quad &k=0,\dots,m-1\\
Z^{0}_{k}r+|\Lambda_{k}^{-1}|(Z_{k}^{1}r+Z_{k}^{2}r^{2}), & k=m,\dots,M-1.
\end{array}\right.
$$
Finally the element $Z_{ M}$ will be defined to satisfy \eqref{eq:coda}.
\vskip 5pt

{\it The bound $Z_{0}$}

\noindent Since $|v_{k}|\leq_{cw} w_{k}^{-s}\mathds 1$, define $Z^{0}$ as
\begin{equation}
 (Z^{0})_{k}=\left\{\begin{array}{ll}
\left[|I-A A^{\dag}|\{w_{j}^{-s}\mathds 1\}_{j=0}^{m-1}\right]_{k}\quad &k=0,\dots,m-1\\
0, & k\geq m
\end{array}\right.
\end{equation} 
so that
$$
\left|\left[I-AA^{\dag} \right]rv\right|\leq_{cw}Z^{0}r .
$$
Note that  $A^{\dag}$ is an almost inverse of $A$, indeed by definition $A_mDf^{(m)}\approx I$. Then the size of $Z^0$ is  small and  depends on the accuracy of the numerical method that computes the inverse $A_m$.  

{\it  The bounds $Z^{1},Z^{2}$}

\noindent Concerning the terms in $Z^{1}, Z^{2}$, consider the expansion as quadratic polynomial in $r$
\begin{equation}\label{eq:expansion}
\left[\left(Df(\bar x+ru)- A^{\dag}\right)rv\right]_{k}=\sum_{i=1,2}c_{k,i}r^{i}.
\end{equation}
First note that
$$
{\displaystyle(A\cdot Q)_{k,1}=\sum_{j+l=k}\Big(Re(\mathcal A_{j})Q_{l,1}-Im(\mathcal A_{j})Q_{l,2} \Big)},
$$
$$
{\displaystyle(A\cdot Q)_{k,2}=\sum_{j+l=k}\Big(Im(\mathcal A_{j})Q_{l,1}+Re(\mathcal A_{j})Q_{l,2} \Big)},
$$
then, taking in mind that $Q_{k,2}=-Q_{-k,2}$ and denoting with $sg(l)={\rm sign}(l)$, one computes

\begin{equation} \label{eq:c0}
\begin{array}{rl}
c_{0,1}=\left[\begin{array}{c}
2\sum_{k\geq m}v_{k},\\
{\displaystyle -\sum_{\substack{
l+j=0\\
|l|\geq m
}}\Big(Re(\mathcal A_{j})-sg(l)Im(\mathcal A_{j})\Big)v_{|l|}}
\end{array}\right] ,\quad c_{0,2}=\left[\begin{array}{c}
0\\
u_{0}v_{0}+v_{0}u_{0}
\end{array}\right],
\end{array}
\end{equation}
for  $ k=1,\dots, m-1$
\begin{equation} \label{eq:ck_1_m}
\begin{array}{rl}
c_{k,1}={\displaystyle -\sum_{\substack{
l+j=k\\
|l|\geq m
}}}\left[\begin{array}{c}
{\displaystyle \Big(Re(\mathcal A_{j})-sg(l)Im(\mathcal A_{j})\Big)v_{|l|}}\\
{\displaystyle \Big(Im(\mathcal A_{j})+sg(l)Re(\mathcal A_{j})\Big)v_{|l|}}
\end{array}\right] ,\quad c_{k,2}=\left[\begin{array}{c}
u_{k}v_{0}+v_{k}u_{0}\\
u_{k}v_{0}+v_{k}u_{0}
\end{array}\right],
\end{array}\
\end{equation}
and for $k\geq m$
\begin{equation} \label{eq:ck_greater_m}
\begin{array}{rl}
c_{k,1}={\displaystyle -\sum_{\substack{
l+j=k\\
|l|\neq k
}}}\left[\begin{array}{c}
{\displaystyle \Big(Re(\mathcal A_{j})-sg(l)Im(\mathcal A_{j})\Big)v_{|l|}}\\
{\displaystyle \Big(Im(\mathcal A_{j})+sg(l)Re(\mathcal A_{j})\Big)v_{|l|}}
\end{array}\right] ,\quad c_{k,2}=\left[\begin{array}{c}
u_{k}v_{0}+v_{k}u_{0}\\
u_{k}v_{0}+v_{k}u_{0}
\end{array}\right].
\end{array}
\end{equation}

Therefore $Z^{1}$, $Z^{2}$ need to be defined so that  $Z^{1}_{k}\geq_{cw}|c_{k,1}|$ and $Z^{2}_{k}\geq_{cw}|c_{k,2}|$.
To achieve this, it is  enough to substitute in the above expression the bounds $|u_{k}|,|v_{k}|\leq_{cw}w_{k}^{-s}\mathds 1_{n}$ and $|\pm Re(\mathcal A_{j})\pm Im(\mathcal A_{j})|\leq_{cw} |Re(\mathcal A_{j})|+|Im(\mathcal A_{j})|$.
Since $\mathds 1_{n}\mathds 1_{n}=n\mathds 1_{n}$, one gets
\begin{equation}\label{eq:Z2k}
\begin{array}{l}
|c_{0,2}|\leq_{cw}2n\left[\begin{array}{c}
0\\
\mathds 1_{n}\\
\end{array}\right]=: Z^{2}_{0},\\
|c_{k,2}|\leq_{cw}2nw_{k}^{-s}\left[\begin{array}{c}
\mathds 1_{n}\\
\mathds 1_{n}\\
\end{array}\right]=:Z^{2}_{k},\ k\geq 1.
\end{array}
\end{equation}

  However this approach is not completely feasible  for the computation of $|c_{k,1}|$, due to  the presence of series. Therefore it is necessary to introduce  further computational parameters 
\begin{equation}\label{eq:Lk}
L_{k}>\max\{k,m\}
\end{equation}
 and matrices  $H_{k}$ so that
\begin{equation}\label{eq:boundH}
\begin{array}{ll}
|c_{0,1}|\leq_{cw}\left[\begin{array}{c}
2\sum_{j= m}^{L_{0}}w_{j}^{-s}\mathds 1_{n},\\
{\displaystyle \sum_{\substack{
l+j=0\\
m\leq |l|\leq L_{0}
}}\Big(|Re(\mathcal A_{j})|+|Im(\mathcal A_{j})|\Big)w_{l}^{-s}\mathds 1_{n}}
\end{array}\right] +H_{0}=:Z^{1}_{0}&\\
|c_{k,1}|\leq_{cw}{\displaystyle \sum_{\substack{
l+j=k\\
m\leq|l|\neq k, |l|\leq L_{k}
}}}\left[\begin{array}{c}
{\displaystyle \Big(|Re(\mathcal A_{j})|+|Im(\mathcal A_{j})|\Big)w_{l}^{-s}\mathds 1_{n}}\\
{\displaystyle \Big(|Im(\mathcal A_{j})|+|Re(\mathcal A_{j})|\Big)w_{l}^{-s}\mathds 1_{n}}\\
\end{array}\right]+H_{k}=:Z^{1}_{k},& k=1,\dots,m-1\\
\end{array}
\end{equation}
and similarly for $k\geq m$. 
It means that the bound $Z^{1}$ has been defined as sum of two factors: the first obtained by rigorous computation of a finite number of elements in the series, the second analytically defined to estimate the tail part of the series that have not been computed. 

Define
$$
\zeta(M,s)\bydef \frac{1}{(M+1)^{s}}+\frac{1}{(M+2)^{s}}+\frac{1}{s-1}\frac{1}{(M+2)^{s-1}},
$$
and 
\begin{equation}\label{eq:H}
H_{0}\bydef\left[\begin{array}{c}
2\zeta(L_{0},s)\mathds 1_{n}\\
h_{0}\mathds 1_{n}
\end{array} \right],\quad H_{k}\bydef h_{k}\left[\begin{array}{l}
\mathds 1_{n}\\
\mathds 1_{n}
\end{array} \right],
\end{equation}
where for $k\geq 0$
$$
h_{k}=\frac{\sqrt{2}n\|\mathcal A\|_{s^{\star}}}{(L_{k}+1-k)^{s^{\star}-s}}\Big(\zeta(L_{k}-k,2s)+\zeta(L_{k},2s)\Big).
$$
Hence, one has the following result.
\begin{lemma}\label{lemma:H}
Formula  \eqref{eq:boundH} holds for $H_{0}$, $H_{k}$ defined in \eqref{eq:H}. 
\end{lemma}
\begin{proof}
First note that for any $M\geq 1$ and $s\geq 2$
\begin{equation} \label{eq:estimate_1}
\sum_{k=M+1}^{\infty}\frac{1}{k^{s}}< \zeta(M,s).
\end{equation}
That can be seen from the fact that
$\sum_{k=M+1}^{\infty}\frac{1}{k^{s}}=\frac{1}{(M+1)^{s}}+\frac{1}{(M+2)^{s}}+\sum_{k=M+3}^{\infty}\frac{1}{k^{s}}<\frac{1}{(M+1)^{s}}+\frac{1}{(M+2)^{s}}+\int_{M+2}^{\infty}k^{-s}dk$.
Hence, one has that
$$
\left|2\sum_{j= L_{0}+1}^{\infty}w_{j}^{-s}\mathds 1_{n}\right|\leq_{cw}2\zeta(L_{0},s)\mathds 1_{n} .
$$
For the remaining terms note that $(|Re(\mathcal A_{j})|+|Im(\mathcal A_{j}|)\leq_{cw}\sqrt{2}|\mathcal A_{j}|\leq_{cw}\sqrt{2}\frac{\|\mathcal A\|_{s^{\star}}}{w_{j}^{s^{\star}}}\mathds 1_{n}$, then, for any $k\geq 0$, the tail part of the series can be bounded by
$$
\left|{\displaystyle \sum_{\substack{
l+j=k\\
 |l|\geq L_{k}+1
}}}
\Big(|Re(\mathcal A_{j})|+|Im(\mathcal A_{j})|\Big)w_{l}^{-s}\mathds 1_{n}\right|\leq_{cw}\sqrt{2}\|\mathcal A\|_{s^{\star}}\sum_{l=L_{k}+1}^{\infty}\left(\frac{1}{w_{k-l}^{s^{\star}}}+\frac{1}{w_{k+l}^{s^{\star}}}\right)\mathds 1_{n}w_{l}^{-s}\mathds 1_{n}
$$
$$
\leq_{cw}\frac{\sqrt{2}n\|\mathcal A\|_{s^{\star}}}{(L_{k}+1-k)^{s^{\star}-s}}\sum_{l=L_{k}+1}^{\infty}\left(\frac{1}{w_{l-k}^{s}}+\frac{1}{w_{k+l}^{s}}\right)w_{l}^{-s}\mathds 1_{n} .
$$
In the last passage we have used the fact that $L_{k}>k$, $s^{\star}\geq s$ and the relation  $\mathds 1_{n}\mathds 1_{n}=n\mathds 1_{n}$.
The result follows by applying \eqref{eq:estimate_1} once the last series has be rewritten as
$$
\begin{array}{ll}
\displaystyle{\sum_{l=L_{k}+1}^{\infty}\left(\frac{1}{w_{l-k}^{s}}+\frac{1}{w_{k+l}^{s}}\right)w_{l}^{-s}}&=\displaystyle{\sum_{l=L_{k}+1}^{\infty}\left(\frac{1}{w_{k+l}^{s}}\right)w_{l}^{-s}+\sum_{l=L_{k}-k+1}^{\infty}\left(\frac{1}{w_{k+l}^{s}}\right)w_{l}^{-s}}\\
&\leq\displaystyle{\sum_{l=L_{k}+1}^{\infty}w_{l}^{-2s}+\sum_{l=L_{k}-k+1}^{\infty}w_{l}^{-2s}}.
\end{array}
$$
\end{proof}
\vskip 5pt

{\it The bound $Z_{ M}$}

\noindent From \cite{GL1}, one has that
$$
\sum _{\substack{
k_1+k_2=k \\
|k_1|\neq k
}}  \frac{1}{w_{k_1}^{s}w_{k_2}^{s}}\leq\frac{1}{w_{k}^{s}}\left[2+2\sum_{l=1}^{M}\frac{1}{l^{s}}+\frac{2}{ M^{s-1}(s-1)}+\eta_{M}-1-\frac{1}{w_{2k}^{s}} \right],
$$
where
$$
\eta_{ k}=2\left[\frac{k}{k-1} \right]^{s}+\left[\frac{4\log(k-2)}{k}+\frac{\pi^{2}-6}{3} \right]\left[ \frac{2}{k}+\frac{1}{2}\right]^{s-2}.
$$
Recall the definition of $c_{k,1}$ and $c_{k,2}$ given in \eqref{eq:expansion}, with a more explicit form in \eqref{eq:ck_greater_m} for the case $k\ge m$. Then for $k\geq  M$ one has that 
$$
\begin{array}{rl}
|c_{k,1}|_{\infty}&\leq\sqrt{2}\|\mathcal A \|_{s^{\star}}\displaystyle{\sum_{\substack{
l+j=k\\
 |l|\neq k}}w_{j}^{-s}w_{l}^{-s}}\\
 &\displaystyle{\leq\frac{\sqrt 2\|\mathcal A \|_{s^{\star}}}{w_{k}^{s}}\left[2+2\sum_{l=1}^{ M}\frac{1}{l^{s}}+\frac{2}{M^{s-1}(s-1)}+\eta_{M}-1\right]=:\frac{\sqrt 2\|\mathcal A \|_{s^{\star}}}{w_{k}^{s}}C_{1}} , \\
	\displaystyle{|c_{k,2}|_{\infty}}&\displaystyle{\leq\frac{2n}{w_{k}^{s}}}.
 \end{array}
$$
Since for $k\geq  M$ the first term on the right hand side of \eqref{eq:splitnorm} is zero, the following estimate holds
\begin{equation}
\begin{array}{ll}
\Big|\left[DT(\bar x+ru)rv\right]_{k}\Big|_{\infty}&\leq_{cw}\Big|\left[ A\left(Df(\bar x+ru)- A^{\dag}\right)rv\right]_{k}\Big|_{\infty}\\
&\leq\|\Lambda_{k}^{-1}\|_{\infty}(|c_{k,1}|_{\infty}r+|c_{k,2}|_{\infty}r^{2}).
\end{array}
\end{equation}
Finally, combining \eqref{eq:asymp_bound} and that $k \ge M$, one gets that $\|\Lambda_{k}^{-1}\|_{\infty}\leq \frac{C_{\Lambda}}{ M}$, and we can define $Z_{ M}$ as
$$
Z_{ M}=\frac{C_{\Lambda}}{ M^{2}}(\sqrt{2}\|\mathcal A \|_{s^{\star}}C_{1}r+2n r^{2}) .
$$

\section{Computing stable and unstable tangent bundles of periodic orbits using Floquet normal forms}  
\label{sec:computing_bundles}
Consider an autonomous differential equation
\begin{equation}\label{eq:syst}
\dot y=g(y), \quad g \in C^{1}(\mathbb R^{n})
\end{equation} 
and suppose that $\gamma(t)$ is a $\tau$-periodic solution with $\gamma(0)=\gamma_{0}$. Denote by $\Gamma=\{\gamma(t),t\in[0,\tau]\}$ the support of $\gamma$ and for any $\theta\in[0,\tau]$, define $\gamma_{\theta}(t)=\gamma(t+\theta)$ the phase-shift re-parametrization of $\Gamma$. Being  autonomous, system \eqref{eq:syst} has the property that any of the curves $\gamma_{\theta}(t)$ is a $\tau$-periodic solution satisfying $\gamma_{\theta}(0)=\gamma(\theta)$. We refer to $\Gamma$ as the periodic orbit and $\gamma_{\theta}$ as the periodic solutions.

\begin{definition}[Monodromy matrix]
Let $\gamma:\mathbb R\rightarrow  \mathbb R^{n}$ be a $\tau$-periodic solution \eqref{eq:syst}
and let $\Phi_{\theta}(t)$ be the unique solution of the non-autonomous linear problem
\begin{equation}
	\left\{\begin{array}{l}
\dot \Phi_{\theta}=\nabla g(\gamma_{\theta}(t))\Phi_{\theta}\\
\Phi_{\theta}(0)=I.
\end{array}\right. 
\end{equation}
The matrix $\Phi_{\theta}(\tau)$ is called the {\em monodromy matrix} of $\gamma_{\theta}(t)$.
\end{definition}
Having chosen $\gamma(t)=\gamma_{0}(t)$, in the following we identify $\Phi(\tau)=\Phi_{0}(\tau)$. The next two Lemmas are classical results and are direct consequence of $\Phi_{\theta}(t)$ being a fundamental matrix solution. For sake of completeness, we present their proofs.

\begin{lemma}\label{lemma-mon}
For any $\theta\in[0,\tau]$, the solution $\Phi_{\theta}(t)$ of \eqref{eq:syst} satisfies
$$
\Phi_{\theta}(n \tau +t)=\Phi_{\theta}(t)\Phi_{\theta}(\tau)^{n},\ \forall t\in\mathbb R, \ \forall n\in \mathbb N . 
$$
\end{lemma}
\begin{proof}
Without loss of generality let us consider $\theta=0$.
By induction on $n\geq 0$. For $n=0$ the result is obvious. Suppose it holds for $n-1$.
Then
$$
\Phi(n \tau )=\Phi((n-1)\tau+\tau)=\Phi(\tau)\Phi(\tau)^{n-1}=\Phi(\tau)^{n}.
$$
Define
$$
\Psi(t)=\Phi(t+n \tau )\Phi(n \tau )^{-1} .
$$
It follows that $\Psi(0)=I$ and that
$$
\dot \Psi(t)=\dot \Phi(n \tau +t)\Phi(n \tau )^{-1}=A(n \tau +t)\Phi(n \tau +t)\Phi(n \tau )^{-1}=A(t)\Psi(t)
$$
For the uniqueness of solutions of the initial value problem, $\Psi(t)=\Phi(t)$ thus
$$
\Phi(t+n \tau )=\Phi(t)\Phi(n \tau )=\Phi(t)\Phi(\tau)^{n}, \ \forall t\in\mathbb R.
$$
\end{proof}

\begin{lemma}\label{lemma-equiv}
The matrices $\Phi_{\theta}(\tau)$ are equivalent under conjugation. In particular
\begin{equation}\label{matr-conj}
\Phi_{\theta}(\tau)=\Phi(\theta)\Phi(\tau)\Phi(\theta)^{-1} .
\end{equation}
\end{lemma}
\begin{proof}
The matrix $\tilde\Phi(t):=\Phi(t+\theta)$ is solution of the equation $\dot y=\nabla g(\gamma(t+\theta))y=\nabla g(\gamma_{\theta}(t))$, with $\tilde\Phi(0)=\Phi(\theta)$. Since $\Phi_{\theta}(t)$ is the principal fundamental solution of the previous system, 
$$
\tilde\Phi(t)=\Phi_{\theta}(t)\Phi(\theta).
$$ 
It follows
\begin{equation}\label{eq:phithetat}
\Phi_{\theta}(t)=\tilde\Phi(t)\Phi(\theta)^{-1}=\Phi(t+\theta)\Phi(\theta)^{-1},\quad \forall t. 
\end{equation}
Thus
$$
\Phi_{\theta}(\tau)=\Phi(\tau+\theta)\Phi(\theta)^{-1}=\Phi(\theta)\Phi(\tau)\Phi(\theta)^{-1}
$$
where, in the last passage, Lemma~\ref{lemma-mon} has been used.
\end{proof}

The previous result implies that all monodromy matrices $\Phi_{\theta}(\tau)$ have the same eigenvalues. That motivates the following definition.

\begin{definition}
The eigenvalues $\sigma_{j}$ of the monodromy matrix $\Phi(\tau)$ are called the {\em Floquet multipliers} of the periodic orbit $\Gamma$.
\end{definition}

As already mentioned in Section~\ref{sec:Intro}, in the theory of dynamical systems the monodromy matrix $\Phi(\tau)$ associated to a periodic solution $\gamma(t)$ plays a fundamental role since it  encompasses the information about the stability character of $\gamma$. Indeed, as shown in Proposition~2.122 in \cite{MR2224508}, the Floquet multipliers of $\gamma(t)$ are in fact the eigenvalues of $D{\bf P}(\gamma(0))$, where ${\bf P}(x)$ denotes the Poincar\'e map of $\gamma(t)$ on  a $(n-1)$-dimensional hypersurface transversal to $\gamma$ at $\gamma(0)$. 
Moreover, it can be proved that at least one of the Floquet multipliers $\sigma_{j}$ of $\Phi(\tau)$ is equal to one, corresponding to the eigenvector $\dot\gamma(0)$. Hence, we will denote by $\sigma_n=1$ the Floquet multiplier corresponding to $\dot\gamma(0)$ and denote by $\{\sigma_{j}\}_{j=1,\dots,n-1}$ the set of non trivial Floquet multipliers. We refer to Section~2.4 in \cite{MR2224508} for a more extensive analysis of the links between Poincar\'e sections and Floquet theory. Based on the above discussion, we are now ready to introduce the definition of stability of a periodic orbit. 

\begin{definition} \label{def:stability}
Let $\Gamma=\{\gamma(t),t\in[0,\tau]\}$ be a $\tau$-periodic orbit of the system \eqref{eq:syst} and let $\{\sigma_{j}\}_{j=1,\dots,n-1}$ be the corresponding set of non trivial Floquet multipliers. We say that 
\begin{itemize}
\item $\Gamma$ is {\em stable} if $\forall ~j \in \{1,\dots,n-1\}$, $|\sigma_{j}|<1$;
\item $\Gamma$ is {\em unstable} if  $\exists ~ j \in \{1,\dots,n-1\}$ such that  $|\sigma_{j}|>1$.
\end{itemize}
Moreover, if $p<n-1$ Floquet multipliers have modulus less than one, and $q<n-p$ Floquet multipliers have modulus greater than one, $\Gamma$ is said to have $p$ stable directions and $q$ unstable directions.
\end{definition}
%

Let us mention that there is a variant to the Floquet normal form introduce in Theorem~\ref{Th-Floquet}, namely there exist a constant (possibly complex) matrix $B$ and a nonsingular (possibly complex) continuously differentiable, $\tau$-periodic matrix function $P(t)$ such that $\Phi(t)=P(t) e^{Bt}$. We refer to Theorem~2.83 in \cite{MR2224508} for more details and for the proof. Therefore, there exists a (possibly complex) matrix B such that $\Phi(\tau)=e^{B\tau}$. Denoting by $\lambda_{j}$ the eigenvalues of $B$, it follows that $\sigma_{j}=e^{\tau \lambda_{j}}$ is a Floquet multiplier. Note that for a given $\sigma_{j}$, the solution $\lambda_{j}$ of $\sigma_{j}=e^{\tau \lambda_{j}}$ is not uniquely defined. Indeed for any $k \in \mathbb Z$,  $e^{\tau (\lambda_{j}+i\frac{2k\pi}{\tau})}=\sigma_{j}$. This reflects the fact that in the {\em complex Floquet normal form} $\Phi(t)=P(t)e^{Bt}$, the matrix $B$ is also not uniquely defined. In the literature it is common to call a Floquet exponent associated to $\sigma_{j}$ any complex number $\lambda_{j}$ so that $\sigma_{j}=e^{\tau \lambda_{j}}$. On the converse, for any $\sigma_{j}$ there is a unique real number $l_{j}$ so that $|\sigma_{j}|=e^{l_{j}\tau}$. That motivates the following definition.

\begin{definition} \label{def:lyapunov_exponents}
A {\em Lyapunov exponent} associated to a Floquet multiplier $\sigma_j$ is the unique real number $l_{j}$ so that $|\sigma_{j}|=e^{l_{j}\tau}$.
\end{definition}

Note that using the notion Lyapunov exponents, a definition of stability of a periodic orbit similar to the one of Definition~\ref{def:stability} can be introduced. Indeed, given a $\tau$-periodic orbit $\Gamma=\{\gamma(t),t\in[0,\tau]\}$ of \eqref{eq:syst} and considering $\{l_{j}\}_{j=1,\dots,n-1}$ to be the corresponding set of non trivial Lyapunov exponents, we say that $\Gamma$ is stable if $l_{j}<0$, $\forall~ j=1,\dots,n-1$ and that $\Gamma$ is unstable if there exists $j \in \{1,\dots,n-1\}$ such that $l_{j}>0$. 

Given a real $n \times n$ diagonalizable matrix $A$, let us introduce the notation $\Sigma(A)=\{\alpha_{k},v_{k}\}_{k=1,\dots,n}$ to denote the eigendecomposition of the square matrix $A$, i.e. $Av_{k}=\alpha_{k}v_{k}$, for all $k=1,\dots,n$. 

The following result shows how the information from the couple $(R,Q(t))$ coming from the Floquet normal form $\Phi(t)=Q(t)e^{Rt}$ can directly be used to study the dynamical properties of the periodic orbit $\Gamma$. More explicitly, it demonstrates that the stability of $\Gamma$ can be determined by the eigenvalues of $R$ while the stable and unstable tangent bundles of $\Gamma$ can be retrieved from the action of $Q(t)$ (with $t \in [0,\tau]$) on the eigenvectors of $R$. 
\begin{theorem}\label{theorem:bundle}
Assume that $\Gamma=\{\gamma(t),t\in[0,\tau]\}$ is a $\tau$-periodic orbit of \eqref{eq:syst} and consider $\Phi(t)$ the fundamental matrix solution of the non-autonomous linear equation $\dot{y}=\nabla g(\gamma(t))y$ such that $\Phi(0)=I$. Suppose that a Floquet normal form decomposition of Theorem~\ref{Th-Floquet}) $\Phi(t)=Q(t)e^{Rt}$ is known. Assume that the real $n \times n$ matrix $R$ is diagonalizable and let $\Sigma(R)=\{\mu_{j},v_{j}\}_{j=1,\dots,n}$ the eigendecomposition of $R$. Then the Lyapunov exponents $l_j$ of $\Gamma$ are given by
\begin{equation} \label{eq:lyapunov_exp}
l_{j}=Re(\mu_{j}) .
\end{equation}
Furthermore, for any $\theta\in[0,\tau]$, if one defines
\begin{equation} \label{eq:eigs_formulas}
w_{j}^{\theta} \bydef Q(\theta)v_{j},
\end{equation}
then $w_{j}^{\theta}$ is an eigenvector of $\Phi_{\theta}(\tau)$ associated to the Lyapunov exponent $l_{j}$. Note that $w_{j}^{\theta}$ is a smooth $2\tau$-periodic function of $\theta$.
\end{theorem} 
\begin{proof}
Consider the eigendecomposition $\Sigma(R)=\{\mu_{j},v_{j}\}_{j=1,\dots,n}$ of the diagonalizable matrix $R$, meaning that the set $\{v_1,\dots,v_n\}$ consists of $n$ linearly independent eigenvectors of $R$. By Lemma~\ref{lemma-mon}, one has that $\Phi(\tau)^{2}=\Phi(2\tau)$. Since $Q(t)$ is $2\tau$-periodic and $Q(0)=I$, it follows that $\Phi(2\tau)=e^{R2\tau}$. Since $R$ is diagonalizable, $\Phi(2\tau)=e^{R2\tau}$ is also diagonalizable. Since $\Phi(2\tau)=\Phi(\tau)^2$ and since the matrix $\Phi(\tau)$ is invertible and defined over the field of complex number (which has zero characteristic), then it can then be showed that $\Phi(\tau)$ is also diagonalizable.  Now, since $\Phi(2\tau)=\Phi(\tau)^{2}$ one has that if $(\sigma,w) \in \Sigma(\Phi(\tau))$, then $(\sigma^2,w) \in \Sigma(\Phi(2\tau))$. Combining this last point with $\Phi(\tau)$, $\Phi(2\tau)$ being diagonalizable implies that the eigenspaces of $\Phi(\tau)$ and $\Phi(2\tau)$ are in one-to-one correspondence. That implies the existence of a set $\{ \sigma_j\}_{j=1,\dots,n}$ such that $\Sigma(\Phi(\tau))=\{\sigma_{j},v_{j}\}_{j=1,\dots,n}$. From the property of the exponential matrix operator, $\Sigma(\Phi(2\tau))=\{e^{\mu_{j}2\tau},v_{j}\}_{j=1,\dots,n}=\Sigma(\Phi(\tau)^2)=\{\sigma_{j}^2,v_{j}\}_{j=1,\dots,n}$. This implies that $\sigma_{j}^2=e^{\mu_{j}2\tau}$ for any $j=1,\dots,n$. Note that $l_j = Re(\mu_j)$ is the unique real number so that $|\sigma_{j}|=e^{l_j\tau}$. Hence, $l_j$ is a Lyapunov exponent associated to the Floquet multipliers $\sigma_j$.

Now, from \eqref{eq:phithetat}, one has that 
$$
\Phi_{\theta}(2\tau)=\Phi(2\tau+\theta)\Phi(\theta)^{-1}=Q(\theta)e^{(2\tau+\theta)R}e^{-R\theta}Q(\theta)^{-1},\quad \forall \theta\in [0,\tau]
$$
thus
$$
\Phi_{\theta}(2\tau)Q(\theta)v_{j}=Q(\theta)e^{2 \tau R}v_{j}=e^{2 \tau \mu_{j}}Q(\theta)v_{j}
$$
showing that $\Sigma(\Phi_{\theta}(2 \tau ))=\{e^{2 \tau \mu_{j}},Q(\theta)v_{j}\}$. Applying the same argument than above, one can conclude that 
$\Sigma(\Phi_\theta(\tau))=\{\sigma_{j},Q(\theta)v_{j}\}_{j=1,\dots,n}$ forms an eigendecomposition of the matrix $\Phi_{\theta}(\tau)$. Hence, 
$w_{j}^{\theta}=Q(\theta)v_{j}$ is an eigenvector of $\Phi_{\theta}(\tau)$.
By the smoothness and the $2\tau$-periodicity of the matrix function $Q(\theta)$, one can conclude that $w_{j}^{\theta}=Q(\theta)v_{j}$ is also a smooth $2\tau$-periodic function of $\theta$.
\end{proof}



Recall \eqref{eq:eigs_formulas} and consider $w_{j}^{\theta}=a_{j}^{\theta}+i b_{j}^{\theta}$. We define the stable and unstable subspaces $E_{s}^{\theta}, E_{u}^{\theta}\subset T_{\gamma(\theta)}\mathbb R^{n} $
 of the periodic orbit $\Gamma$ at the point $\gamma(\theta)$ as
$$
\begin{array}{l}
E_{s}^{\theta}=Span\big\{a_{i}^{\theta},b_{i}^{\theta}: |\sigma_{j}|<0 \big\}\\
\\
E_{u}^{\theta}=Span\big\{a_{i}^{\theta},b_{i}^{\theta}: |\sigma_{j}|>0 \big\} .
\end{array}
$$
That allows us to define the following
\begin{definition} \label{def:bundles}
We define the {\em stable} and {\em unstable tangent bundles} of $\Gamma$ respectively by
$$
E_{s},E_{u}\subset T_{\Gamma} \mathbb R^{n}
$$
$$
E_{s}=\bigcup_{\theta\in[0,\tau]}\{\gamma(\theta)\}\times E_{s}^{\theta},\quad E_{u}=\bigcup_{\theta\in[0,\tau]}\{\gamma(\theta)\}\times E_{u}^{\theta}.
$$
\end{definition}
It is important to remark that from the conclusion of Theorem~\ref{theorem:bundle}, the complete structure of the stable and unstable bundles can be recovered by the action of the matrix function $Q(t)$ on the eigenvectors of $R$, which themselves correspond to the stable and unstable directions at the point $\gamma(0)$ on  $\Gamma$. Also, the proof of Theorem~\ref{theorem:bundle} is constructive in the sense that combined with the rigorous computational method of Section~\ref{sec:rig_comp}, it provides a computationally efficient direct way to obtain the eigenvectors $w_{j}^{\theta}$ of $\Phi_\theta(\tau)$, which are the ingredients defining the bundles of Definition~\ref{def:bundles}. Note that one could be tempted to use the fact that $\Phi(\tau)=Q(\tau)e^{R \tau}$ and then attempt to compute the eigendecomposition of $\Phi(\tau)$ directly. However, that would imply having to compute the exponential of an interval valued matrix, which turns out to be a difficult task (e.g. see \cite{goldsztejn}, \cite{MR960775}). This being said, the rigorous computation of the eigendecomposition of the interval matrix $R$ is not completely straightforward. We addressed this problem by adapting the computational method based on the radii polynomials in order to enclose all the solution $\{\mu_{k}, v_{k}\}$ of the nonlinear problem $(R-\mu I)v=0$ with constrain $|v|^{2}=1$. Further details on the enclosure of the eigendecomposition of interval matrices are postponed on a future work of the same authors \cite{enclosure}.

\section{Applications} \label{sec:applications}

In this section, we present some applications, where we construct rigorously tangent stable and unstable bundles of some periodic orbits of the Lorenz equations in Section~\ref{sec:lorenz_bundles} and of the $\zeta^3$-model in Section~\ref{sec:zeta3_bundles}. Note that all rigorous computations were performed in {\em Matlab} with the interval arithmetic package {\em Intlab} \cite{Ru99a}.

\subsection{Bundles of periodic orbits in the Lorenz equations} \label{sec:lorenz_bundles}
Consider the following three dimensional system of ODEs, known as the Lorenz equations 
\begin{equation}\label{syst:lor}
\left\{
\begin{array}{l}
\dot u_{1}=\sigma (u_{2}-u_{1})\\
\dot u_{2}=\rho u_{1}-u_{2}-u_{1}u_{3}\\
\dot u_{3}=u_{1}u_{2}-\beta u_{3}
\end{array} \right.
\end{equation}
with the classical choice of parameters  $\beta=8/3,\sigma=10$ and $\rho$ left as a bifurcation parameter. Suppose to have rigorously proved the existence of  a real $\tau_{\gamma}$-periodic solution $\gamma(t)=[\gamma^{1},\gamma^{2},\gamma^{3}](t)$ of \eqref{syst:lor} in the form
\begin{equation} \label{eq:gamma_lorenz}
\gamma^{j}(t)=\sum_{k\in \mathbb Z}\xi_{k}^{j}e^{ik\frac{2\pi}{\tau_{\gamma}}t},\quad j=1,2,3 
\end{equation}
in a ball of radius $r_{\gamma}$ and centered at $[\bar \tau_{\gamma}, \bar \xi_{k}]$, $|k|\leq M_{\gamma}$, with respect to the $\Omega^{s^{\star}}$ norm,
meaning that
\begin{equation}\label{eq:boundsxi}
\begin{array}{c}
|\tau_{\gamma}-\bar \tau_{\gamma}|\leq r_{\gamma}\\
\\
|Re(\xi_{k})-Re(\bar \xi_{k})|_{\infty}\leq r_{\gamma}w_{k}^{-s^{\star}},\quad |Im(\xi_{k})-Im(\bar \xi_{k})|_{\infty}\leq r_{\gamma}w_{k}^{-s^{\star}},\quad |k|=0,\dots, M_{\gamma}\\
\\
|Re(\xi_{k})|_{\infty}\leq r_{\gamma}w_{k}^{-s^{\star}},\quad |Im(\xi_{k})|_{\infty}\leq r_{\gamma}w_{k}^{-s^{\star}},\quad |k|>M_{\gamma}
\end{array}
\end{equation}
for a decay rate $s^{\star}\geq 2$. Note that $\xi_{k}\in \mathbb Z^{3}$ and $\xi_{-k}=\mathcal C(\xi_{k})$. The existence of such solution could be achieved by applying a modified version of the method discussed in the previous section. Even with some technical differences, the philosophy  is the same. Rewrite the system of ODEs as a infinite dimensional algebraic system where $\tau_{\gamma}$ and the Fourier coefficients $\xi_{k}$ are the unknowns, then consider a finite dimensional projection and compute a numerical approximate solution $\bar \tau_{\gamma}, (\bar \xi_{k})_k$. Then, by means of the radii polynomials, prove the existence, in a suitable Banach space, of a genuine solution $\tau_{\gamma},(\xi_k)_k$ of the infinite dimensional problem in a small ball containing the approximate solution. Note that this is not the first time that the radii polynomials are used to prove existence of periodic solutions of differential equations (e.g. see \cite{MR2443030}, \cite{KL}, \cite{time_periodic_PDEs}, \cite{MR2630003}, \cite{MR2592879}).

In the following we aim to combine the rigorous computational method of Section~\ref{sec:rig_comp} together with Theorem~\ref{theorem:bundle} to rigorously compute the stable and unstable tangent bundles of the periodic orbit $\gamma(t)$ given by \eqref{eq:gamma_lorenz}. This first requires the computation of the fundamental matrix solution of the linearized system along $\gamma(t)$, that is the solution  for $t\in[ 0,\tau_{\gamma}]$ of the non-autonomous system
\begin{equation}\label{syst:y}
\left\{\begin{array}{l}
\dot \Phi=\nabla g(\gamma(t))\Phi\\
\Phi(0)= I
\end{array}
\right.
\end{equation}
where $g$ is the right hand side of \eqref{syst:lor}, $\nabla g$ denotes the Jacobian of the right hand side of system \eqref{syst:lor} and $I$ is the $3\times 3$ identity matrix. The former system is nothing more than a particular case of \eqref{eq:syst-A}, where $A(t)=\nabla g(\gamma(t))$ and $n=3$. We now apply the computational method presented in Section~\ref{sec:rig_comp} to compute the principal fundamental matrix solution of the non-autonomous linear system $\dot{y}=\nabla g(\gamma(t)) y$. 
In particular a constant matrix $R$ and the Fourier coefficients $Q_{k}$ of a $2 \tau _{\gamma}$-periodic function $Q(t)$ will be computed, so that 
$$
\Phi(t)=Q(t)e^{Rt}
$$
is the unique solution of \eqref{syst:y}. Once the computation of $R$ and the $Q_k$ is done, following the conclusion of Theorem~\ref{theorem:bundle}, we will compute $\Sigma(R)=\{ (\mu_j,v_j)~|~j=1,\dots,n\}$, derive from the Lyapunov exponents $l_j \bydef Re(\mu_j)$ the stability of the periodic orbit $\Gamma$ and from the eigenvectors $\{v_1,\dots,v_n\}$ of $R$ we will construct the tangent bundles as defined in Definition~\ref{def:bundles} and given by the formula \eqref{eq:eigs_formulas}.

\vskip 5pt

{\it Computation of $R$ and $Q_{k}$}

\noindent To begin with, let us explicitly write the Jacobian 
 $$
 \nabla g(u)=\left[ \begin{matrix} 
      -\sigma & \sigma & 0 \\
      \rho-u^{3} & -1 & -u^{1} \\
      u^{2} & u^{1} & -\beta
   \end{matrix}\right]
 $$
 and, as consequence, the coefficients $\mathcal A_{k}$
 
 $$
 \mathcal A_0=\left[ \begin{matrix}
      -\sigma & \sigma & 0 \\
      \rho-\xi_{0}^3 & -1 & -\xi^1_0 \\
      \xi^2_0 & \xi^1_0 & -\beta
   \end{matrix}\right]
 , \quad \mathcal A_k=\left[ \begin{matrix} 
      0 & 0 & 0 \\
      -\xi^3_k & 0 & -\xi^1_k \\
      \xi^2_k & \xi^1_k & 0
   \end{matrix}\right], \quad k\geq 1.
 $$
 The hypothesis \eqref{eq:boundsxi} for $\xi_{k}$ to lie in a ball centered at $\bar \xi_{k}$ implies that $\|\mathcal A\|_{s^{\star}}<\infty$. Although this bound is sufficient to proceed with the computational process, we want to stress out that precise informations are known about the $|\mathcal A_{k}|_{\infty}$ of the tail elements  of the sequence $\{\mathcal A_{k}\}$. Indeed it can be easily seen  that
 \begin{equation}\label{eq:boundAk}
 |\mathcal A_{k}|_{\infty}\leq \sqrt 2 r_{\gamma}\frac{1}{w_{k}^{s^{\star}}},\quad  \forall k>M_{\gamma} .
 \end{equation}
 The computation of the approximate solution $\bar R$, $\bar Q_{k,1}$, $\bar Q_{k,2}$ has been addressed as follow: consider the approximation $\bar \gamma(t)=\sum_{|k|\leq M_{\gamma}}\bar \xi_{k}e^{ik2\pi t/ \bar \tau_{\gamma}}$ of the periodic orbit $\gamma(t)$ and numerically solve  system \eqref{syst:y} up to time $2\bar \tau_{\gamma}$. Denote by $\bar y(2\bar \tau_{\gamma})$ the obtained result and  numerically compute  
 $$
 \mathcal R=\log(\bar y(2\bar \tau_{\gamma})).
 $$ 
 Neglect the imaginary part and consider only the real part.
 Then numerically integrate  the  system \eqref{eq:ivp} up to time $2\bar \tau_{\gamma}$ with $\mathcal R$ in place of $R$ yielding the solution $\mathcal Q(t_{j})$.
 Fix the positive finite dimensional parameter $m$ and compute  from $\mathcal Q(t_{j})$ the matrices  $\mathcal Q_{k,1}$, $\mathcal Q_{k,2}$, respectively  the real and imaginary part of the Fourier coefficients with $|k|<m$.  Finally the vector $(\mathcal R, \mathcal Q_{k})$ is considered as starting point for a Newton iteration scheme applied on the finite dimensional reduction defined generally in \eqref{eq:finite_dim_reduction}. Denote the output of the iterative process by $\bar{x}=(\bar R, \bar Q_{k})$, that is an approximated solution $f^{(m)}(\bar x)\approx 0$ up to a desired accuracy, where $f^{(m)}$ is defined in \eqref{eq:finite_dim_reduction}.
 
Consider $\Lambda_k$ given by \eqref{eq:Lamnda_k}. Note that in the case of the three-dimensional vector field \eqref{syst:lor}, $\Lambda_k$ is a $6 \times 6$ matrix and one could compute its inverse analytically using the mathematical software {\em Maple}. 
After having computed $\Lambda_{k}^{-1}$ one needs to check that the chosen $m$ satisfies $m>K$ where $K$ is the same as in Lemma~\ref{lemma:Lambda_invertible}, otherwise increase $m$.
 
 Then for a choice of $M>m$ and $2\leq s\leq s^{\star}$ one can compute the coefficients $Y_{k}$ and $Z_{k}$, $k=0,\dots ,M$  and $Y_{M}, Z_{M}$ as shown in Section~\ref{sec:construction}. It only remains to define the computational parameters $L_{k}$ introduce in \eqref{eq:Lk}. In the computation presented here $L_{k}$ has been chosen as 
 \begin{equation}\label{eq:choiceLk}
 L_{k}=\max\{M+M_{\gamma}\}+k .
 \end{equation}
 This choice assures that the tail elements $H_{0}$, $H_{k}$ in \eqref{eq:boundH} only contain the terms $\mathcal A_{j}$ satisfying $|\mathcal A_{j}|_{\infty}\leq\sqrt{2}r_{\gamma}w_{j}^{-s^{\star}}$. Therefore the subsequent estimate for $h_{k}$ can be improved by replacing $\|\mathcal A\|_{s^{\star}}$ with $r_{\gamma}$.
 
Again, the knowledge of the particular behavior of the coefficients $\mathcal A_{k}$ allows to provide a better estimate for $Z_{M}$.
Indeed note that $|Re(\mathcal A_{k})|\leq_{cw}  |Re(\bar{\mathcal{A}}_{k})|+w_{k}^{-s^{\star}}\mathds 1_{n}$, where $\bar{\mathcal A}_{k}$ denotes the matrix $\mathcal A_{k}$ with the entries $\bar \xi$ in place of $\xi$ and the same holds for $|Im(\mathcal A_{k})|$. Therefore $|Re(\mathcal A_{k})|_{\infty}+|Im(\mathcal A_{k})|_{\infty}<\sqrt{2}|\bar \xi_{k}|_{\infty}+2r_{\gamma}w_{k}^{-s^{\star}}$ for $1\leq|k|\leq M_{\gamma}$ and \eqref{eq:boundAk} for $|k|>M_{\gamma}$.
 
 Therefore, the computation of the bound for $|c_{k,1}|_{\infty}$ when $k\geq M$, necessary  for the definition of $Z_{M}$, has been slightly  modified as follows.
 \begin{equation}
 \begin{array}{ll}
 |c_{k,1}|&=\left|{\displaystyle \sum_{\substack{
l+j=k\\
 |l|\neq k
 }}}
\Big(|Re(\mathcal A_{j})|+|Im(\mathcal A_{j})|\Big)w_{l}^{-s}\mathds 1_{n}\right|\leq_{cw}{\displaystyle \sum_{\substack{
l+j=k\\
j\neq 0,2k
 }}}\Big(|Re(\mathcal A_{j})|+|Im(\mathcal A_{j})|\Big)w_{l}^{-s}\mathds 1_{n}\\
 &\leq{\displaystyle \sum_{\substack{
l+j=k\\
j\neq 0\\
|j|\leq M_{\gamma}
 }}}\Big(|Re(\bar{\mathcal A_{j}})|+|Im(\bar{\mathcal A_{j}})|\Big)w_{l}^{-s}\mathds 1_{n}+2r_{\gamma}{\displaystyle \sum_{\substack{
l+j=k\\
|l|\neq k
 }}}w_{j}^{-s}w_{l}^{-s}\mathds 1_{n}\mathds 1_{n} .
  \end{array}
 \end{equation}
 Then, passing to the infinity absolute value, for any $k\geq M$ 
  \begin{equation}
 \begin{array}{ll}
|c_{k,1}|_{\infty} &\leq_{cw}3\sqrt{2}{\displaystyle \sum_{j=1}^{M_{\gamma}}}|\bar \xi_{j}|_{\infty}(w_{k-j}^{-s}+w_{k+j}^{-s})+2nr_{\gamma}{\displaystyle \sum_{\substack{
l+j=k\\
|l|\neq k
 }}}w_{j}^{-s}w_{l}^{-s}\\
 &\leq_{cw}{\displaystyle \frac{3}{k^{s}}\left[\sqrt{2}\sum_{j=1}^{M_{\gamma}}|\bar \xi_{j}|_{\infty}k^{s}(w_{k-j}^{-s}+w_{k+j}^{-s})+ 2r_{\gamma}\left[1+2\sum_{l=1}^{ M}\frac{1}{l^{s}}+\frac{2}{M^{s-1}(s-1)}+\eta_{M}\right]  \right]}\\
\noalign{\vskip 2pt}&\leq_{cw}{\displaystyle \frac{3}{k^{s}}\left[\sqrt{2}\sum_{j=1}^{M_{\gamma}}|\bar \xi_{j}|_{\infty}\left(\frac{1}{\left(1-\frac{j}{M}\right)^{s}}+1\right)+ 2r_{\gamma}\left[1+2\sum_{l=1}^{ M}\frac{1}{l^{s}}+\frac{2}{M^{s-1}(s-1)}+\eta_{M}\right]  \right]}.
 \end{array}
 \end{equation}
 
 {\it Computational results}

For the choice $\sigma=10 , \beta=8/3$ it is known that there exists  a branch of periodic solutions parametrized by $\rho$ joining a Hopf bifurcation at $\rho=\frac{470}{19}\approx 24.736 $ and a homoclinic point at $\rho\approx 13.9265 $.  Figure~\ref{Fig:orbits}(a-b) shows the bifurcation graph and some of the periodic orbit  of the continuous family.

\begin{figure}
\centering
\subfigure[]{
\hspace{-5pt}\includegraphics[width=0.33\textwidth, height=0.33\textheight]{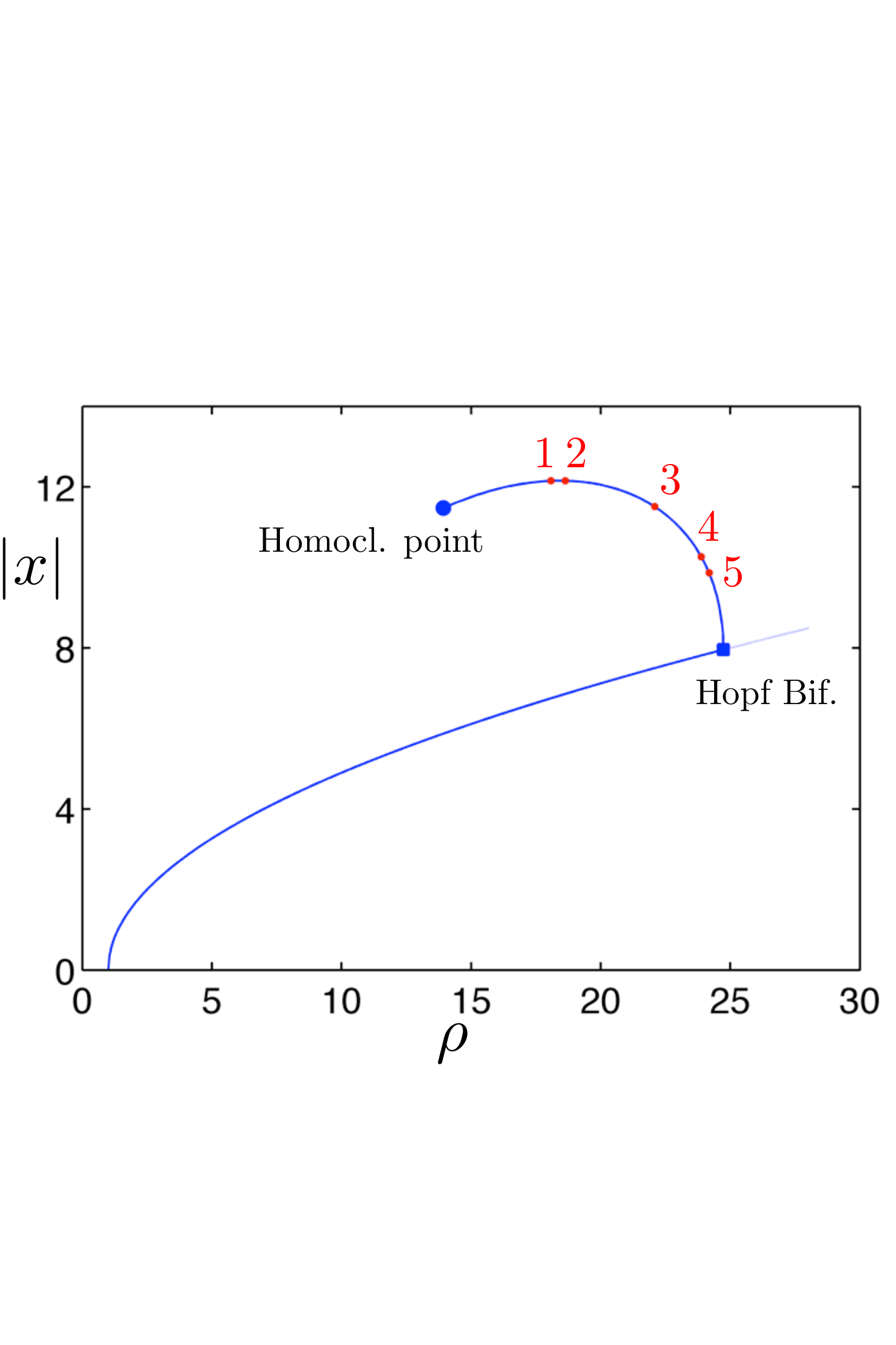}
\label{Fig:orbitsa}}
\subfigure[]{
\includegraphics[width=0.3\textwidth, height=0.3\textheight]{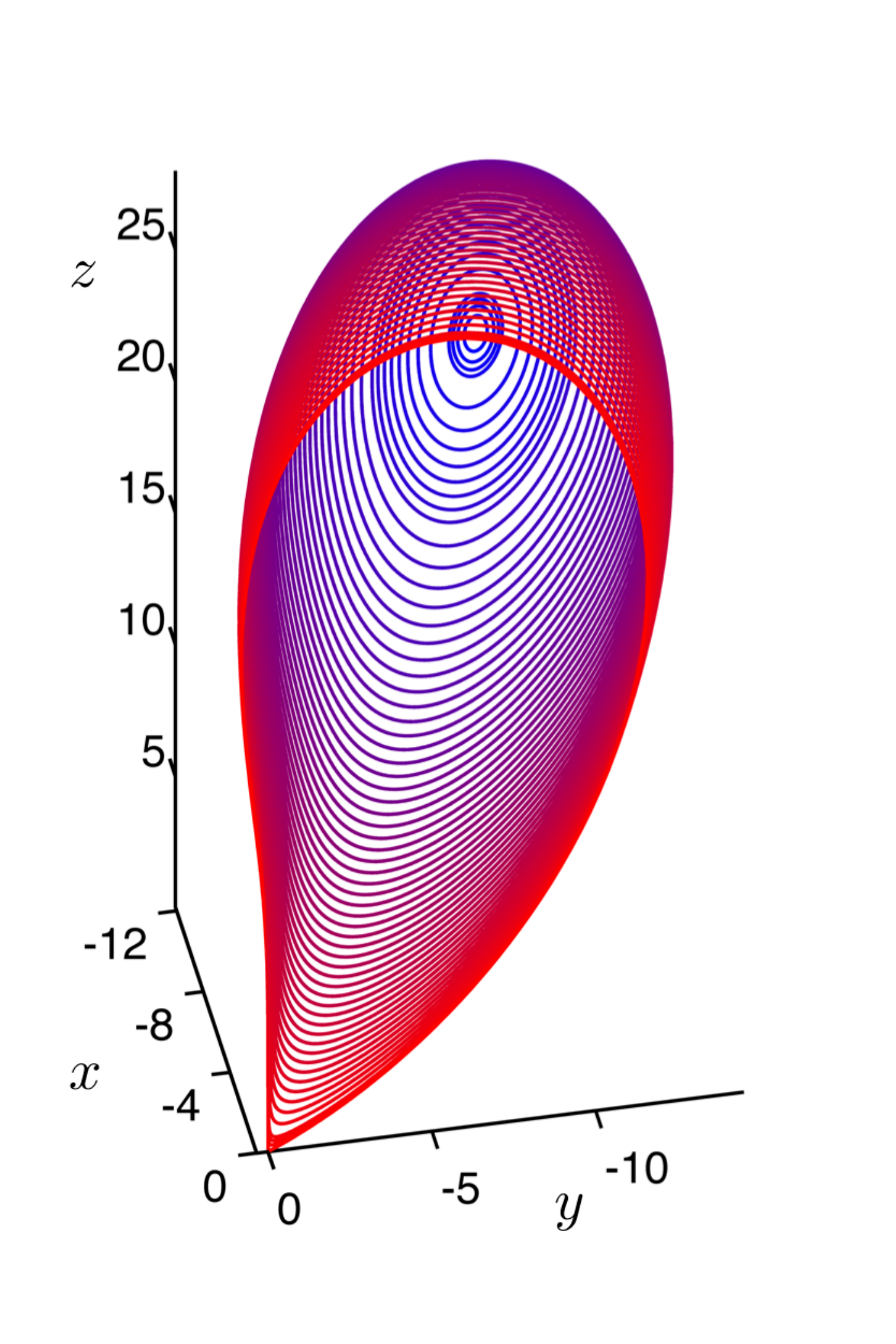}
\label{Fig:orbitsb}}
\subfigure[]{
\includegraphics[width=0.3\textwidth, height=0.3\textheight]{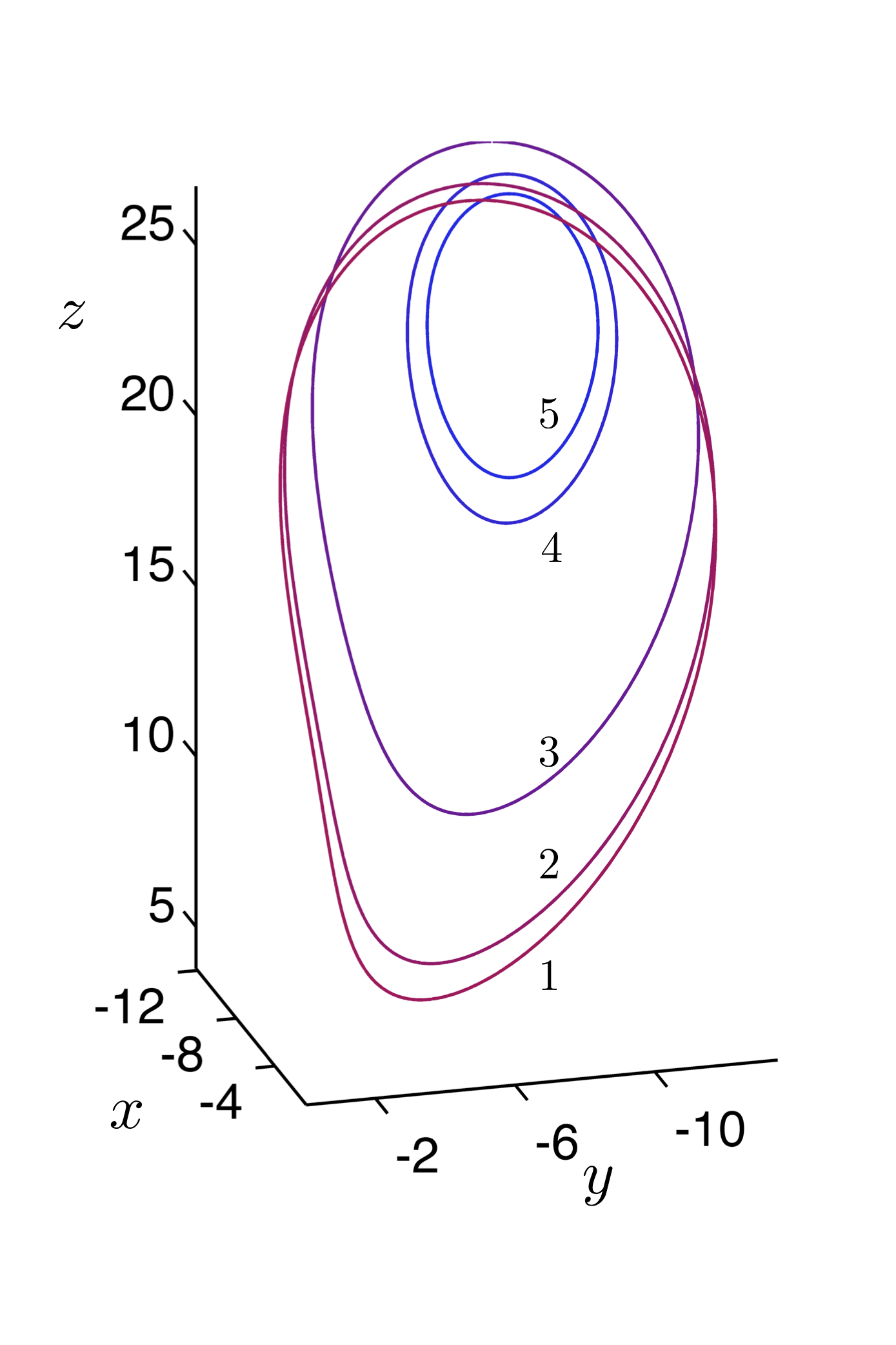}
\label{Fig:orbitsc}}
\caption{{\small (a) The partial bifurcation diagram for the Lorenz system. The labelled points correspond to the values $\rho_{i}$. (b) Some of the periodic orbits on the family joining the Hopf Bifurcation and the homoclinic point. (c) The periodic solutions corresponding to $\rho=\rho_{i}$.}   }
\label{Fig:orbits}
\end{figure}
The computation of the rigorous enclosure of the periodic orbits and successively of their tangent bundles  have been performed for a set of different periodic orbits of the Lorenz system lying on the mentioned bifurcation branch and corresponding to values of $\rho$
$$
\begin{array}{c}
\rho_{1}=18.0815,\ \rho_{2}=18.6815\\
\\
\rho_{3}=20.8815,\ \rho_{4}=23.8815,\ \rho_{5}=24.1816\\
\end{array} 
$$
Figure~\ref{Fig:orbitsc} reports the graphics of the numerical approximation $\bar \gamma_{i}$ of the orbits corresponding to the choice  $\rho_{i}$, while the Table \ref{table:periodic} contains the computational parameter $M_{\gamma}$  that have been chosen for the  rigorous enclosure of the orbit, the period $\bar \tau_{\gamma}$ and the radius  $r_{\gamma}$ resulting from the computations. The growth  rate $s^{\star}$ has been fixed $s^{\star}=2$ for all the cases.

\begin{table}[htdp]
\begin{center}
\begin{tabular}{c|c|c|c}
\# \ {\rm sol}& $M_{\gamma}$ & $\bar \tau_{\gamma}$ &$ r_{\gamma}$\\
\hline
1& 32  & 1.027854 & $6.844864508150837e-09$ \\
2 & 30 & 0.978271 &$7.151582969846857e-09$\\
3& 26 &0.822883 & $4.260379031142465e-09 $\\
4 & 20 & 0.683813 & $5.368959115576269e-09$\\
5& 30 & 0.672595 &$2.360935240171144e-08$\\
\end{tabular}
\end{center}
\caption{{\small The $i$-th row concerns the computation of the periodic orbit for the Lorenz system corresponding to $\rho=\rho_{i}$. $M_{\gamma}$ is the finite dimensional reduction parameter chosen in the computation,  $\bar \tau_{\gamma}$ the approximated period of the solution and $r_{\gamma}$ the resulting enclosing radius.}  }
\label{table:periodic}
\end{table}
In the Appendix the first 15 Fourier coefficients of $\bar \gamma_{1}$ and $\bar\gamma_{4}$ are listed. As shown in Figure~\ref{Fig:decayxik}, one can notice that the Fourier coefficients of the five orbits under consideration are decaying to zero with a different speed. This is due to the fact that the closer we are to the homoclinic orbit, the {\em flatter} the periodic solution is, meaning that a larger number of Fourier coefficients contributes to the Fourier expansion, hence leading to a slower decay.
\begin{figure}
\centering
\includegraphics[scale=0.5]{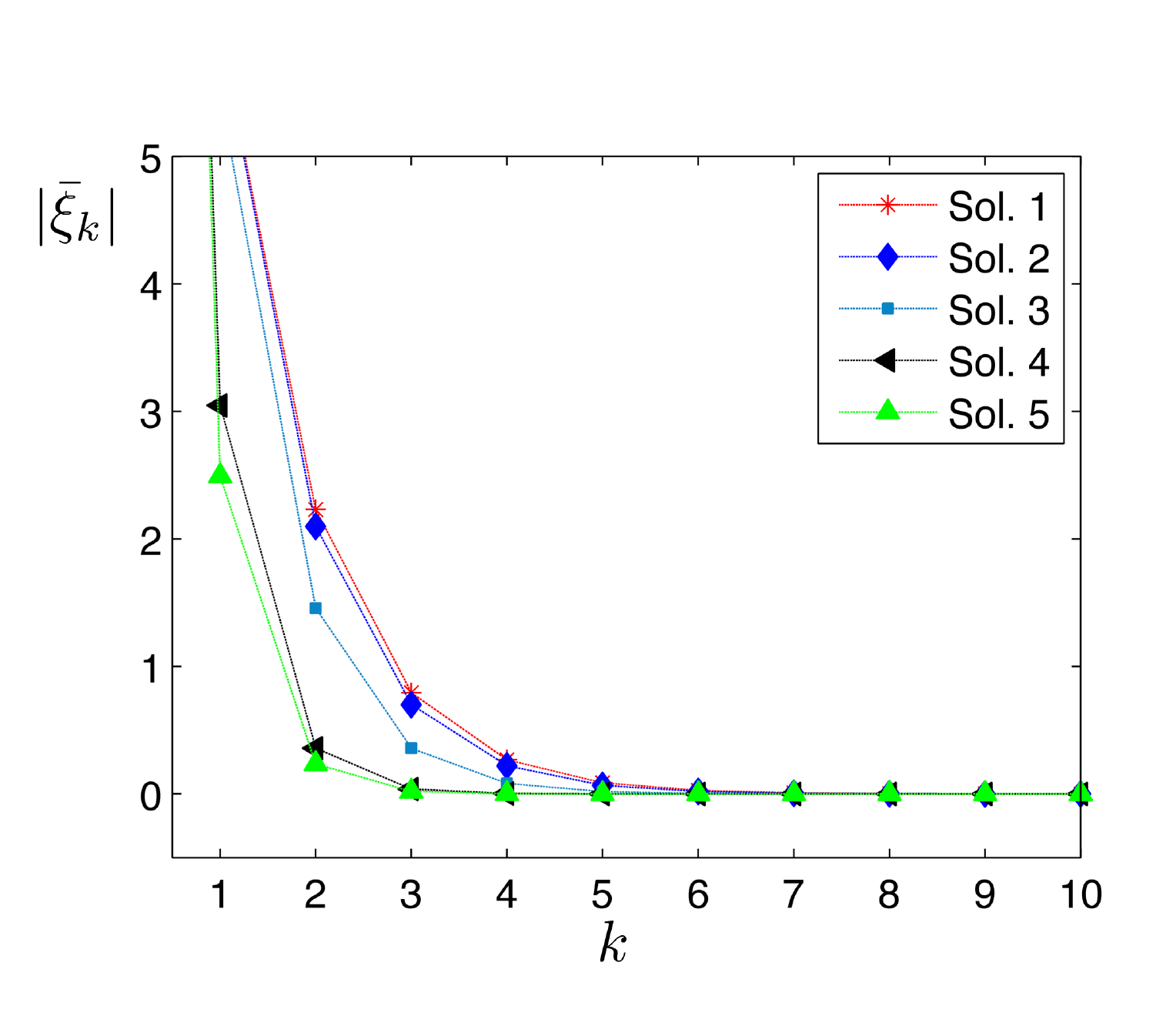}
\vspace{-10pt}\caption{{\small Norm of the Fourier coefficients of each of the periodic solutions $\gamma_{i}$}}
\label{Fig:decayxik}
\end{figure}
Table \ref{Table:floquet} contains information about the computation of the fundamental matrix solution associated to each of the previous periodic orbits. More precisely, it contains the finite dimensional reduction parameter $m$, the computational parameter $M$ and the resulting radius $r$.
\begin{table}[ht]
\begin{center}
\begin{tabular}{c|c|c|c}
\# \ {\rm sol}& m & M & r\\
\hline
1& 100 & 180 & $1.98645943e-05$\\
2 & 90 & 140 & $7.52145121e-06$\\
3& 80 & 80 & $9.66152623e-07 $\\
4 & 60 & 66 & $9.91268997e-07$  \\
5& 60 & 70 & $3.77687574e-06$\\
\end{tabular}
\end{center}
\caption{{\small Computing the fundamental matrix solution for each of the periodic orbit $\gamma_{i}$. $m$ and $M$ are respectively the finite dimensional reduction  parameter and the computational parameter that have been chosen. $r$ is the radius of the ball centered approximate solution in $\Omega^{s}$ within which a genuine solution of  \eqref{eq:four} exists.}}
\label{Table:floquet}
\end{table}%

Some of the radii polynomials $p_{k}(r)$ built during  the computation of solution $\#4$ have been plotted in Figure~\ref{Fig:radii}. The bold line on the $x$-axis remarks the interval
$$
INT =[ 9.91268997\cdot 10^{-7} ,\quad 1.4574858482\cdot 10^{-3}]
$$
where all the radii polynomials are negative.
\begin{figure}
\centering
\includegraphics[width=0.43\textwidth,height=0.21\textheight]{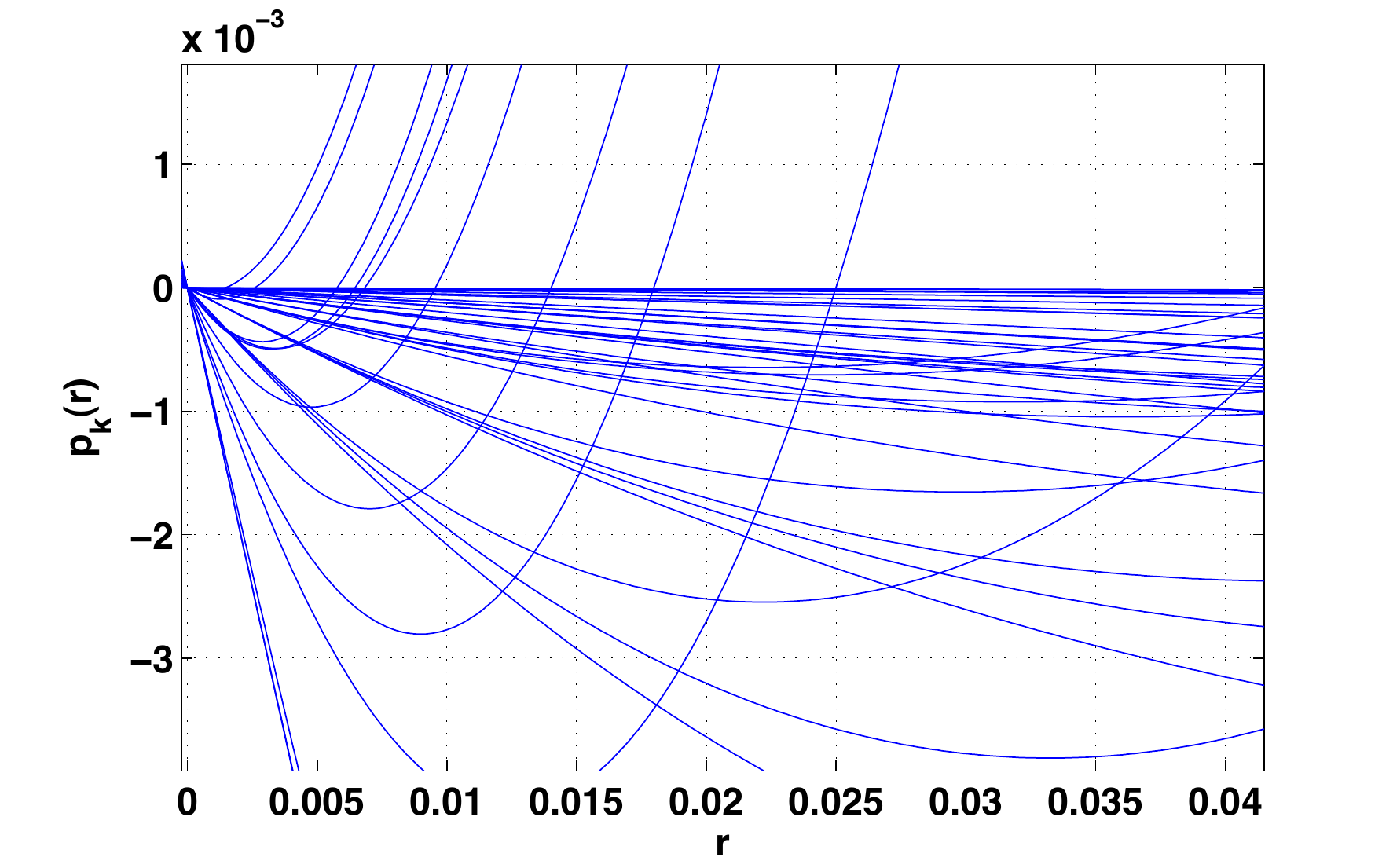}
\vspace{-3pt }\includegraphics[width=0.43\textwidth,height=0.21\textheight]{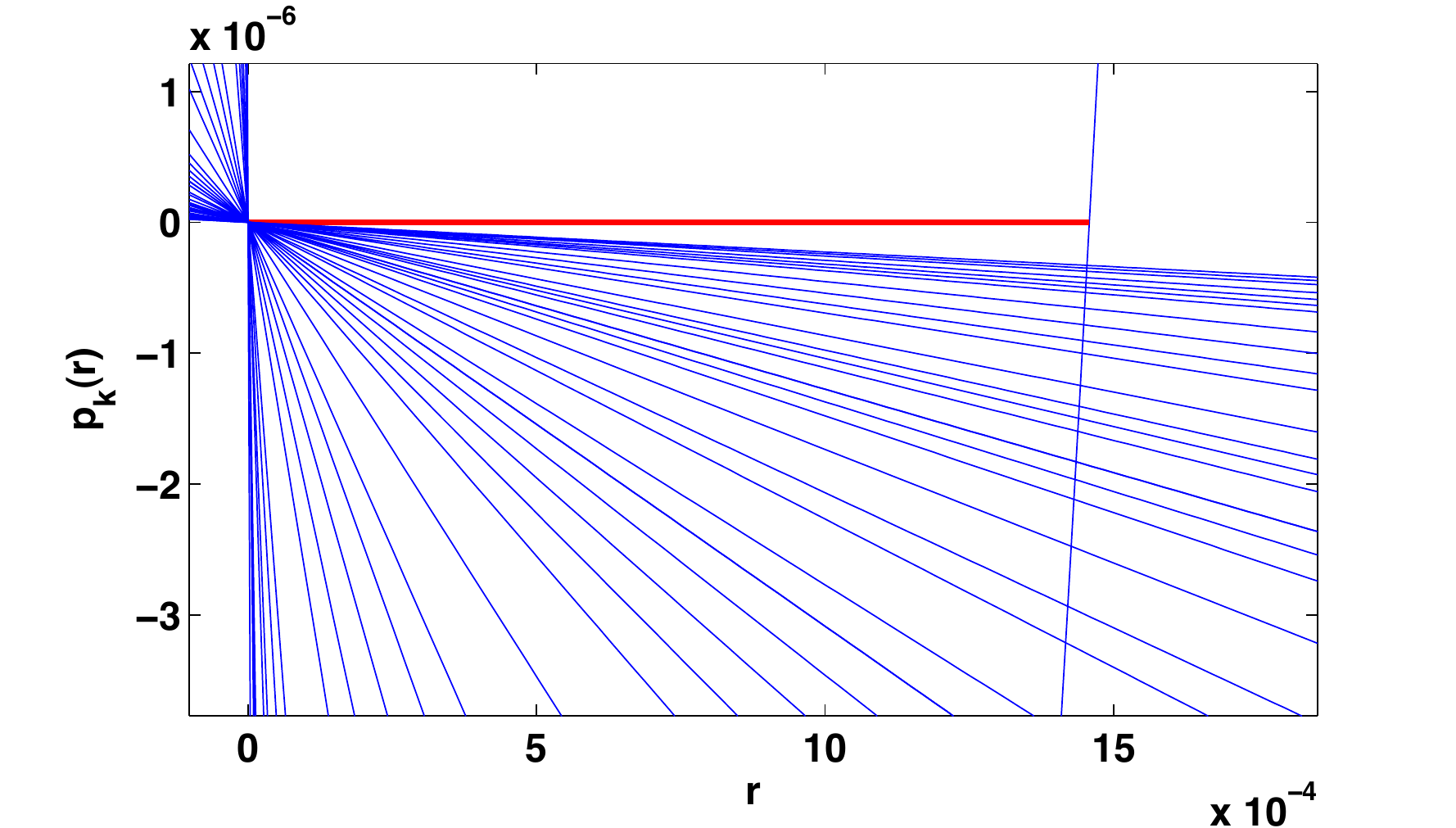}
\caption{{\small Plot of some of the radii polynomials $p_{k}(r)$ constructed in the computation of the fundamental matrix solution associated to $\gamma_{4}$. On the right: magnification close to $r=0$. The red line denotes the interval $INT$ where all the $p_{k}(r)$ are negative.  }}
\label{Fig:radii}
\end{figure}

From the computations we noticed that  the odd Fourier coefficients of $ Q(t)$ are almost vanishing, suggesting that $Q(t)$ is a $\tau_{\gamma}$ periodic function, rather than $2\tau_{\gamma}$ periodic.  This is not in contradiction with Floquet Theorem. Again in the Appendix we report the numerical approximation $\bar R$ and the first even Fourier coefficients $\bar Q_{k}$ for the solution $\#1$ and solution $\#4$.  As in the previous case, the Fourier coefficients $\bar Q_{k}$ corresponding to periodic orbits closer to homoclinic decrease slower. This justifies the fact that larger values of $m$ and $M$ were necessary to obtain successful computations.  

We now have all the ingredients necessary to construct  the tangent bundles: first we compute the intervals  containing  the spectrum and  the eigenvectors of each the interval value matrix $R$, then,  in light of Theorem \ref{theorem:bundle}, the multiplication of the   stable and unstable directions with the function $Q(\theta)$ yields the  {\it tube}  enclosing  the complete stable and unstable bundles. 

Table \ref{table:bundle} lists the Lyapunov exponents of the periodic orbits, as defined in Definition~\ref{def:lyapunov_exponents}, and it also contains the radius of the intervals enclosing the stable and unstable eigen-couple of $R$ while in Figure~\ref{fig:bundles} the tangent bundles are depicted.  In Appendix the complete list of the eigen-decomposition of the interval matrices $R$ is also provided.

\begin{table}[htdp]
\begin{center}
\begin{tabular}{r|rr}
 Sol \# & {\rm Center} & {\rm Radius} \\
 \hline 
 1& -14.2953855130260  &$ 6.801248614\cdot 10^{-5}$\\
  & 0.6287188463595  & $9.510853040\cdot 10^{-4}$\\
 \hline2& -14.2174898849454 &$ 2.814282939\cdot 10^{-5}$\\ 
 & 0.5508232182790 & $3.733964559\cdot 10^{-4}$\\ 
 \hline 3&-13.9620493680589& $2.774785811 \cdot 10^{-6}$ \\
  & 0.2953827013923 & $3.653168667 \cdot 10^{-5}$\\
  \hline 4& -13.7210150091049 & $2.544262339\cdot 10^{-6}$\\
  &  0.0543483424385& $5.248456341\cdot 10^{-5}$\\  
   \hline 5&-13.7013292393391 & $9.720262854 \cdot 10^{-6}$\\
   & 0.0346625726730 & $3.336819199  \cdot 10^{-4}$\\
\end{tabular}
\end{center}
\caption{{\small Lyapunov exponents for each of the periodic orbit $\gamma_{i}$. For each solution we report the center and the radius of the interval vectors enclosing the exponents. Note that we could prove the existence of the eigenvectors $v_j$ associated to $\mu_j$ within the same accuracy given by $r$.}}
\label{table:bundle}
\end{table}

\begin{figure}
\centering
\vspace{-10 pt}
\subfigure[]{
\includegraphics[width=0.32\textwidth]{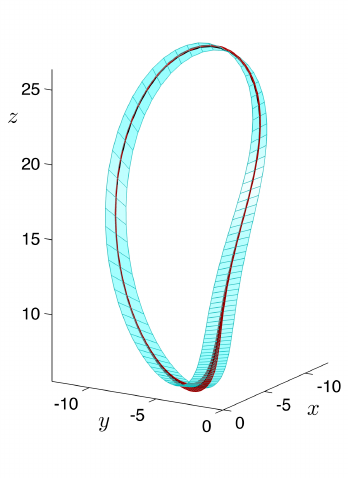}
}\hspace{60 pt}
\subfigure[]{
\includegraphics[width=0.32\textwidth]{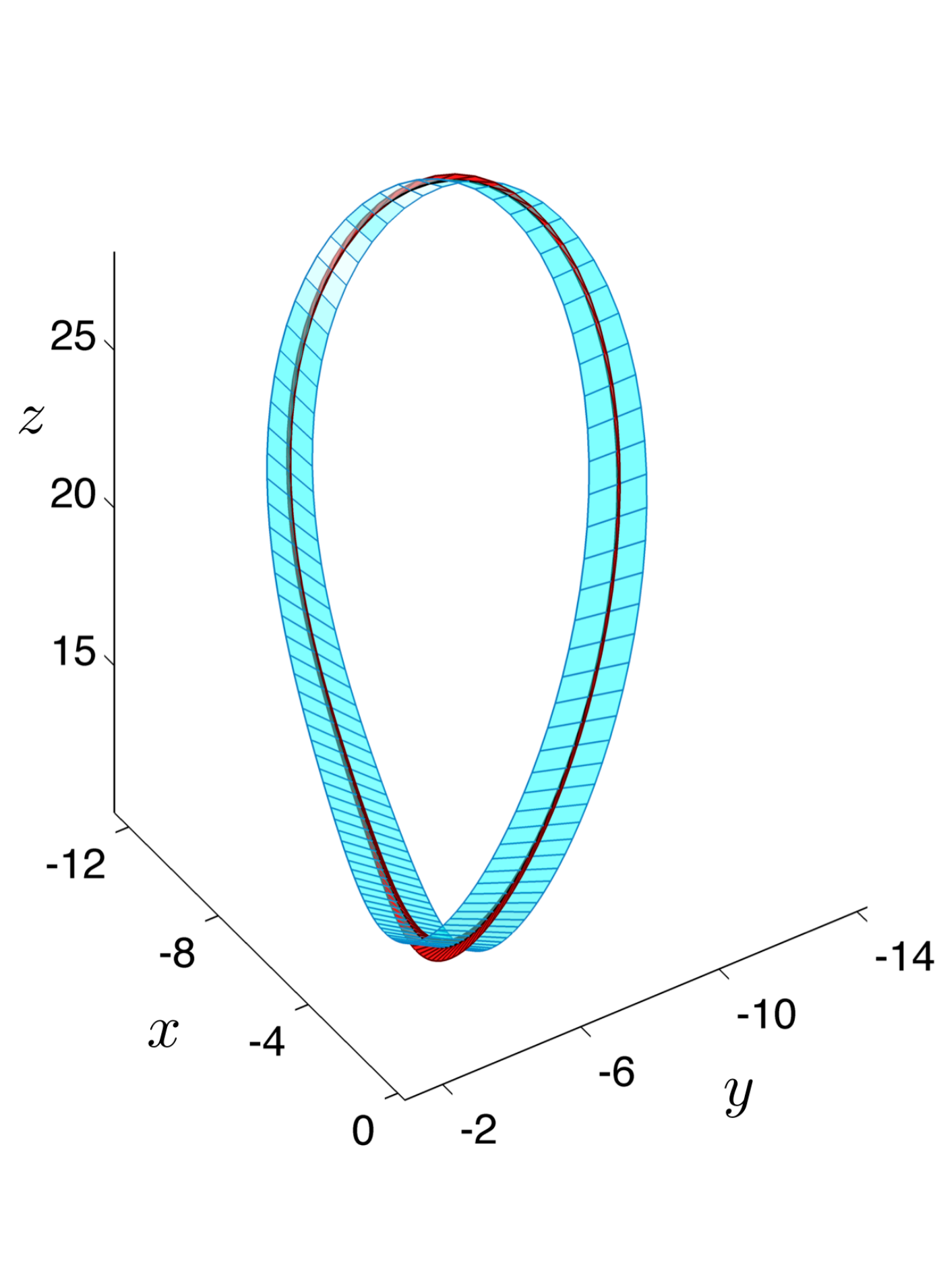}
}
\\\vspace{-35 pt}
\subfigure[]{
\includegraphics[width=0.45\textwidth]{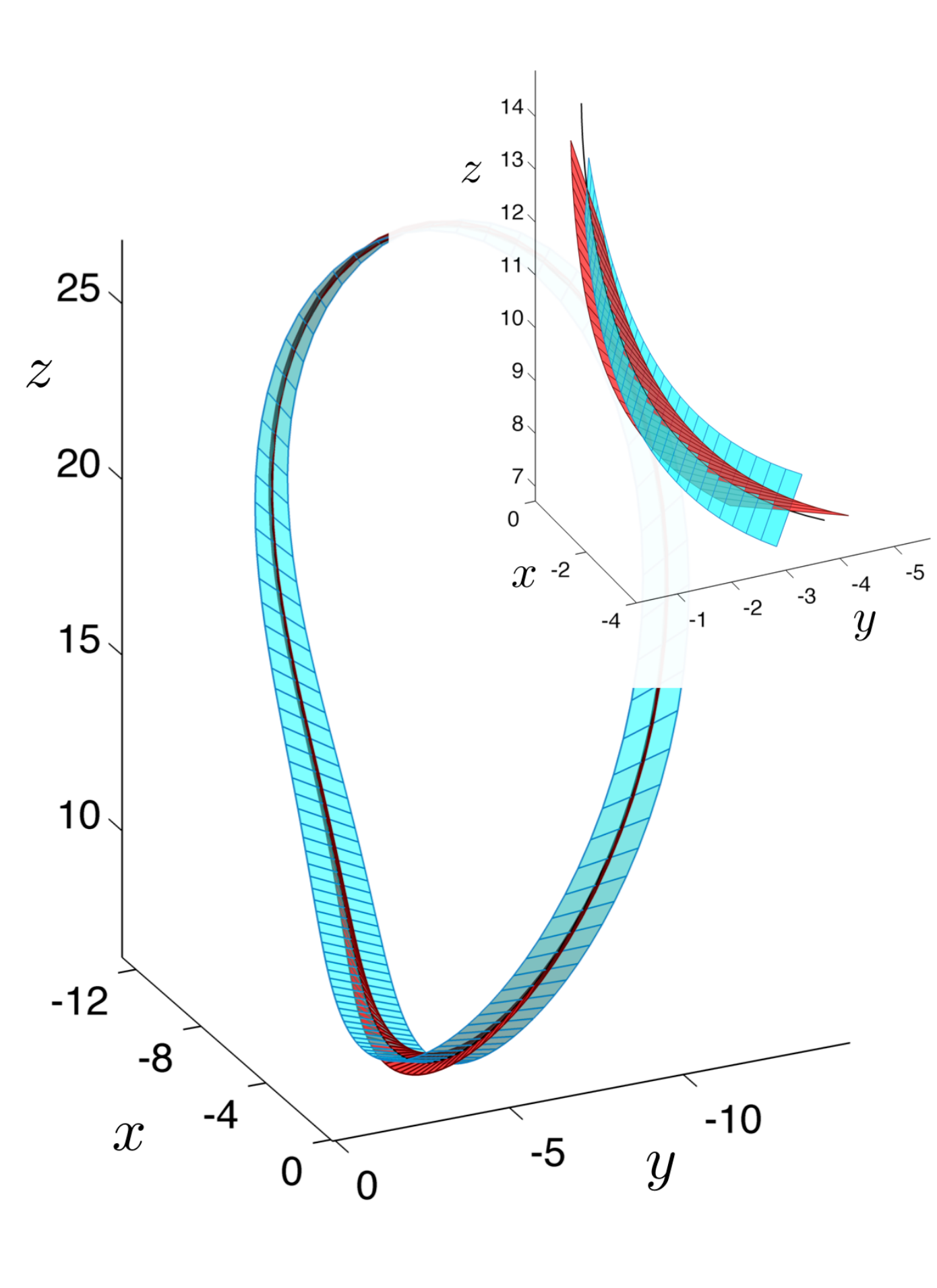}
}
\\\vspace{-35 pt}
\centering\subfigure[]{
\includegraphics[width=0.38\textwidth]{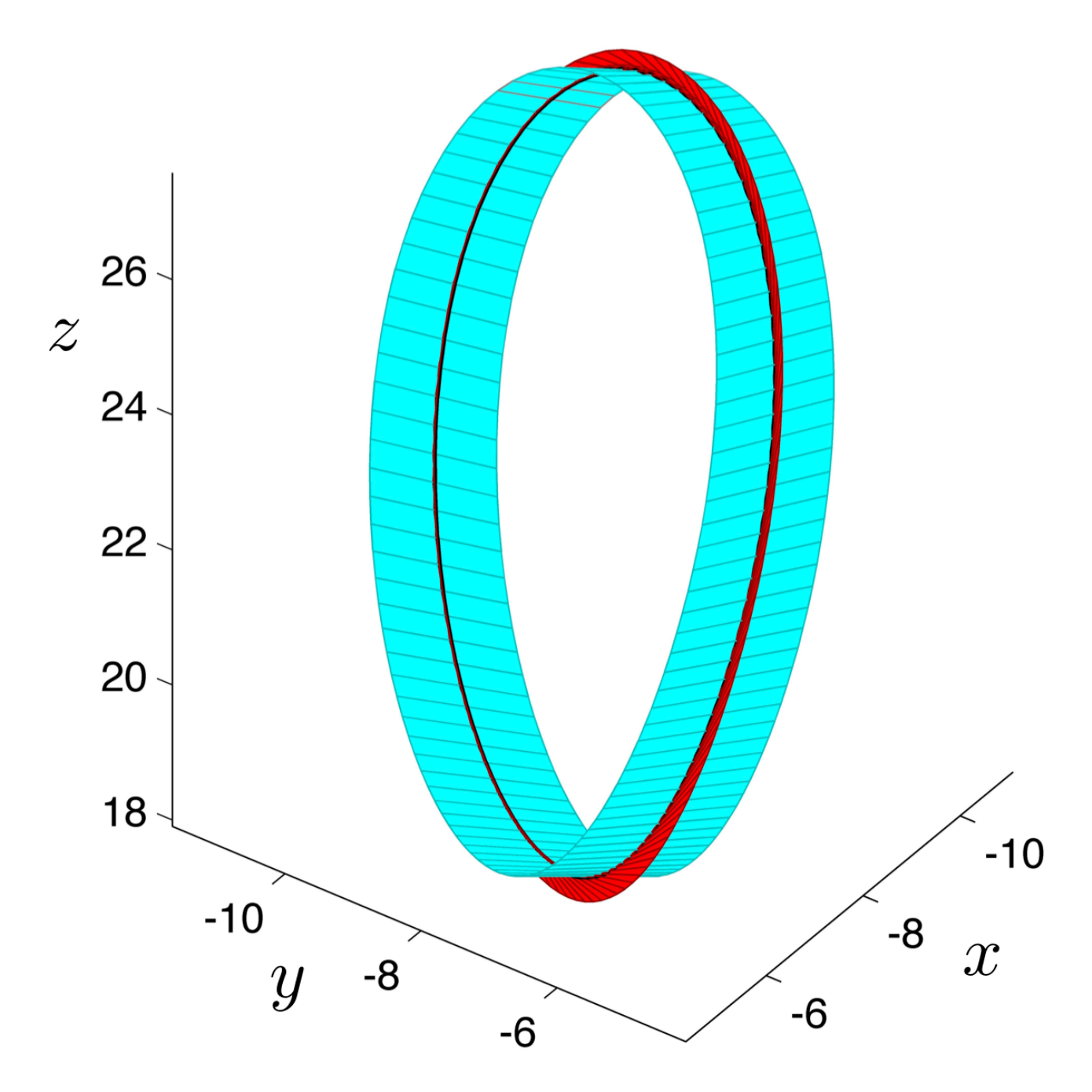}
}\hspace{50 pt}
\subfigure[]{
\includegraphics[width=0.38\textwidth]{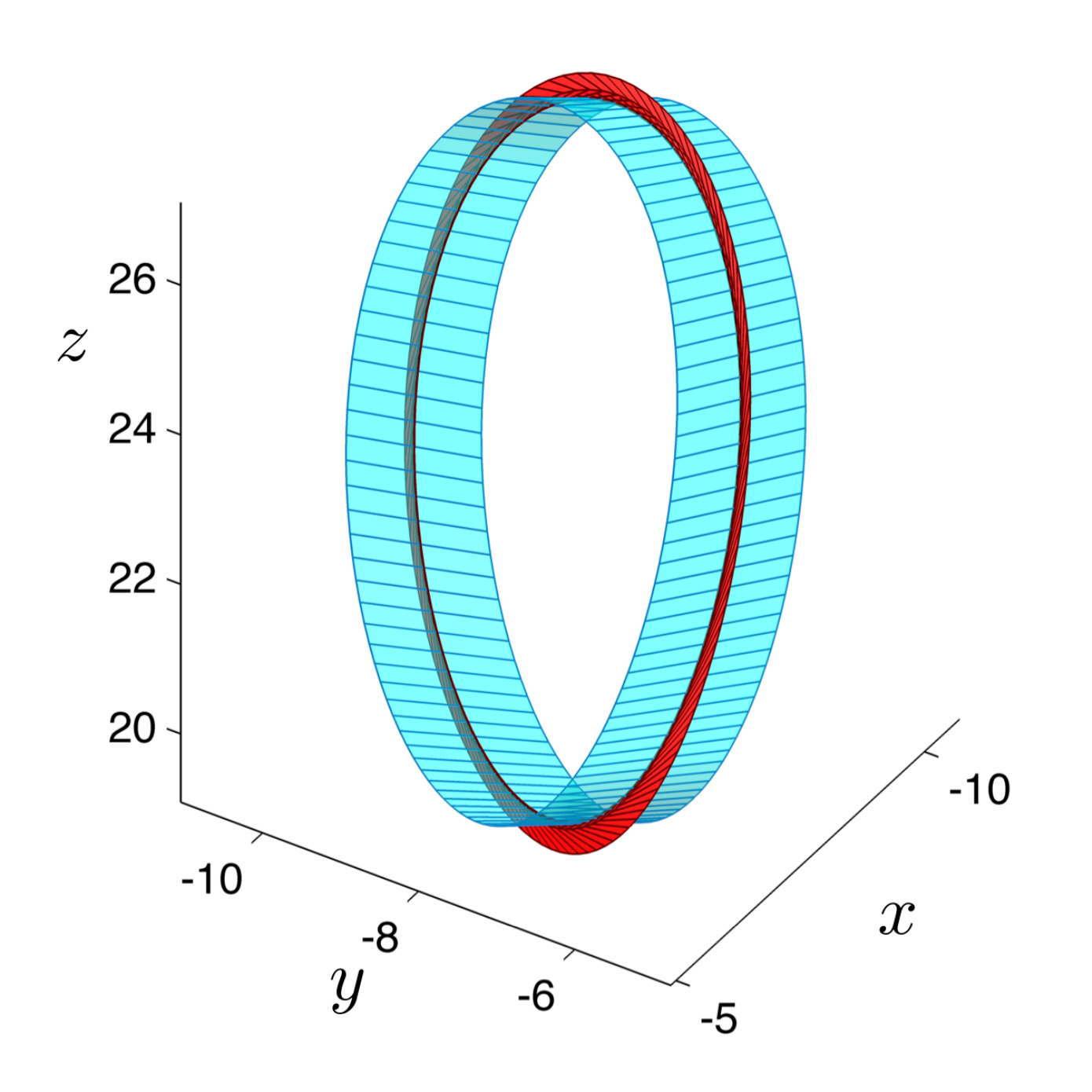}
}
\caption{Plot of the tangent stable (turquoise) and unstable (red) bundles of each of the periodic orbits $\gamma_{i}$. The central figure concerns  Sol$\#$ 2, with a magnification, while figures (a)-(b)-(d)-(e) concern respectively $\gamma_{1}$, $\gamma_{3}, \gamma_{4}, \gamma_{5}$.  }
\label{fig:bundles}
\end{figure}


\subsection{$\zeta^{3}$-model: non orientable tangent bundles} \label{sec:zeta3_bundles}

It is known that if a  Floquet multipliers of a periodic orbit is negative, then the corresponding tangent bundle is not orientable. Moreover,  in the case of a saddle periodic orbit of a three-dimensional system, the two  non-trivial Floquet multipliers are real and their product  is positive. Therefore both the tangent bundles are either orientable or not orientable and, in the latter case, they are topologically equivalent to a M\"obius strip, see \cite{MR1981055}.   

An example of a dynamical system with  periodic orbits that exhibit  this behavior is the so called $\zeta^{3}$-model considered in \cite{MR793709}
\begin{equation}\label{syst:arneodo}
\left\{\begin{array}{l}
\dot x=y\\
\dot y=z\\
\dot z=\alpha x-x^{2}-\beta y-z.
\end{array}
\right.
\end{equation}
For $\beta=2$, as $\alpha$ varies, the periodic orbits of system \eqref{syst:arneodo} produce an interesting bifurcation diagram. We refer to \cite{MR1981055} and \cite{MR1456498} for a detailed analysis of the bifurcation diagram and on the genesis of periodic orbits, called twisted periodic orbit, with non orientable invariant manifolds.  We focus on a particular twisted periodic orbit corresponding to $\alpha=3.372$ lying on the branch  emanating from a period-doubling bifurcation that occurs at  $\alpha\approx 3.125$.

Following the same procedure as before, we rigorously compute the enclosure of the periodic orbit $\gamma(t)$ and subsequently  the enclosure of the  matrix $R$ and of the matrix function $Q(t)$, hence producing an explicit Floquet normal form as in  \eqref{eq:Floquet_normal_form}. Then, we extract the  necessary  stability parameters and we recover the stable and unstable tangent bundles using \eqref{eq:eigs_formulas}.  Figure~\ref{fig:arneodo} shows the resulting  bundles.
  
Having computed  the intervals enclosing the period $\tau$ of the orbit and the  eigenvalues of $R$,  we realize that the absolute values of the  two nontrivial Floquet multipliers satisfy
$$
\begin{array}{l}
|\sigma_{1}|\in[  7.037235782193\cdot 10^{-3}\quad 7.037944324307\cdot 10^{-3} ]\  (=\Delta_{st})\\
|\sigma_{2}|\in[ 1.526609276443494\quad \quad1.528421395487018]\  (=\Delta_{unst})\\
\end{array}
$$

To conclude we emphasize the role played by the continuous function $Q(\theta)$ in the construction of the tangent bundles. As proved in Theorem \ref{theorem:bundle}, as $\theta$ changes, the eigenvector $w_{j}^{\theta}$ of $\Phi_{\theta}(T)$ associated to the Floquet multiplier $\sigma_{j}$ is given by $w_{j}^{\theta}=Q(\theta)v_{j}$, where $v_{j}$ is the eigenvector of  $R$ relative to the eigenvalue $\mu_{j}$. The function $Q(\theta)$ is  continuous and $2\tau$-periodic, but the tangent bundles are  smooth manifolds. That implies that $w_{j}^{\tau}=Q(\tau)v_{j}$ has to be an eigenvector of $\Phi(\tau)$ associated to the Floquet multiplier $\sigma_{j}$, i.e. $span\{v_{j}\}=span\{w_{j}^{\tau}\}$. In the case  of the Lorenz system $Q(\tau)$ turns to be  the identity matrix, therefore the last relation is simply verified. But in case of the $\zeta^{3}$-model and in general when the bundles are not orientable $Q(\tau)$ need not be the identity matrix.   Indeed, in the considered example,   $Q(\tau)$ results to stay  in a small interval around
$$
\bar Q=\left[
\begin{array}{rrr}
-1.663148705259924  &-1.018593776943882  &-0.446703151142258\\
   1.323227706784005   &1.032472501143805   &0.891338521225762\\
   0.936264065136750   &1.438097884773345  &-0.369323795883880\\
 \end{array}\right]\ .
$$
Denoting by $\bar R$, $\bar\tau$, $\bar v_{1}$ the centers of the intervals  the genuine $R$, $\tau$ and $v_{1}$ belong to and defining $\bar\Phi=\tilde Q(\tau)e^{\bar R \bar T}$ the numerical approximation of $\Phi(T)$, we compute 
$$
\bar\Phi\bar Q\bar v=\left[\begin{array}{r}
-0.002642211417990\\
  -0.003294408641461\\
   0.011434535907588\\
   \end{array}\right],\quad\bar Q\bar v=\left[\begin{array}{r}
   0.375442644619642\\
   0.468116019931565\\
  -1.624780050494352\\
    \end{array}\right]
$$
The component-wise ratio between the two computed vectors is $\bar\sigma_{1}=-7.037590044326\cdot 10^{-3}\pm 10^{-13}$, whose absolute value is indeed in the interior of $\Delta_{st}$. If the unstable eigenvector $v_{2}$ is considered, the same operations produce $\bar\sigma_{2}=-1.527515067244305\pm 10^{-13}$. Although not rigorous, these computations confirm the above theoretical discussion  and moreover provide a method to recover the sign of the Floquet multipliers, information that is not possible to achieve following  the presented computational technique.

\begin{figure}
\centering
\subfigure[]{
\includegraphics[scale=0.42]{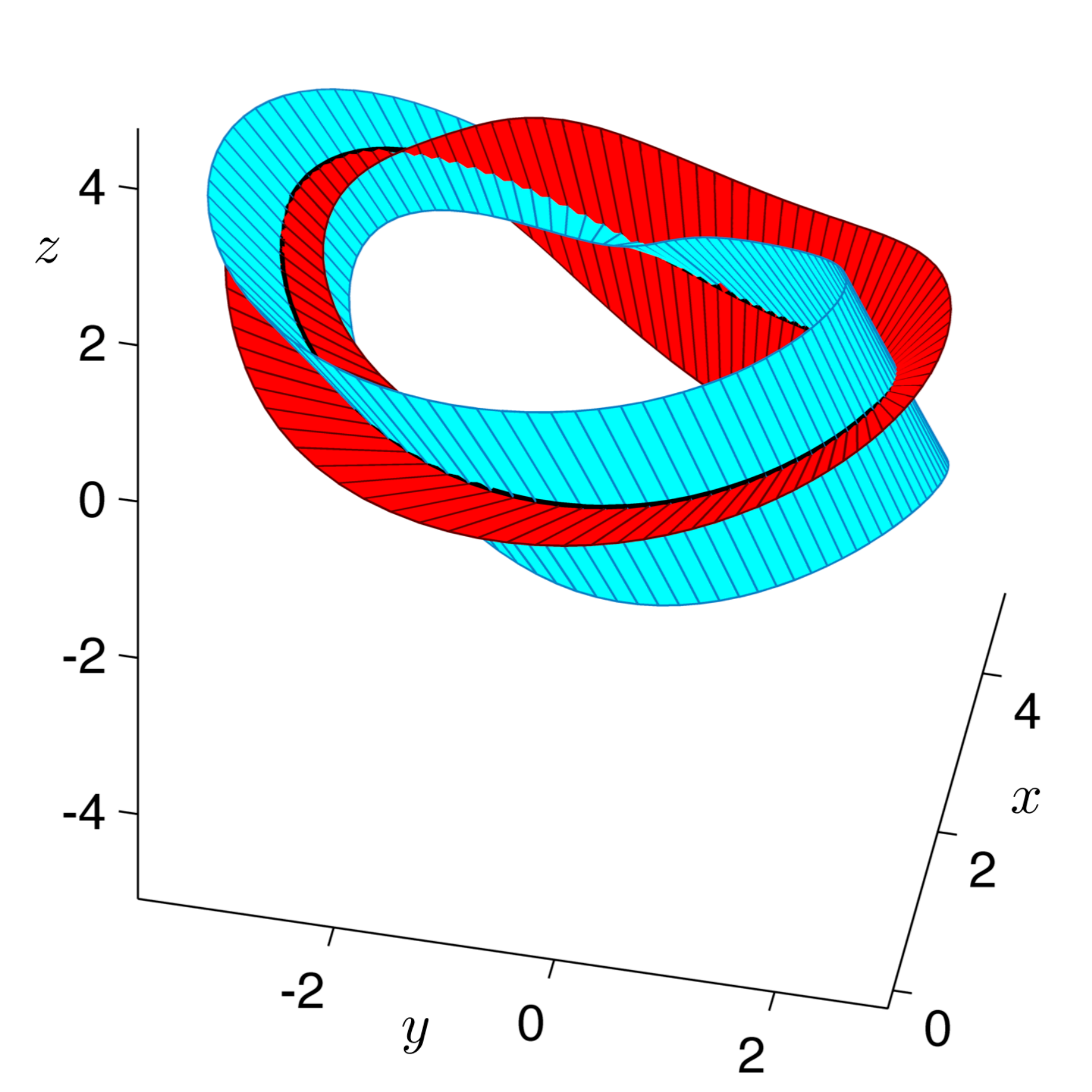}}
\\
\subfigure[]{
\includegraphics[width=0.47\textwidth]{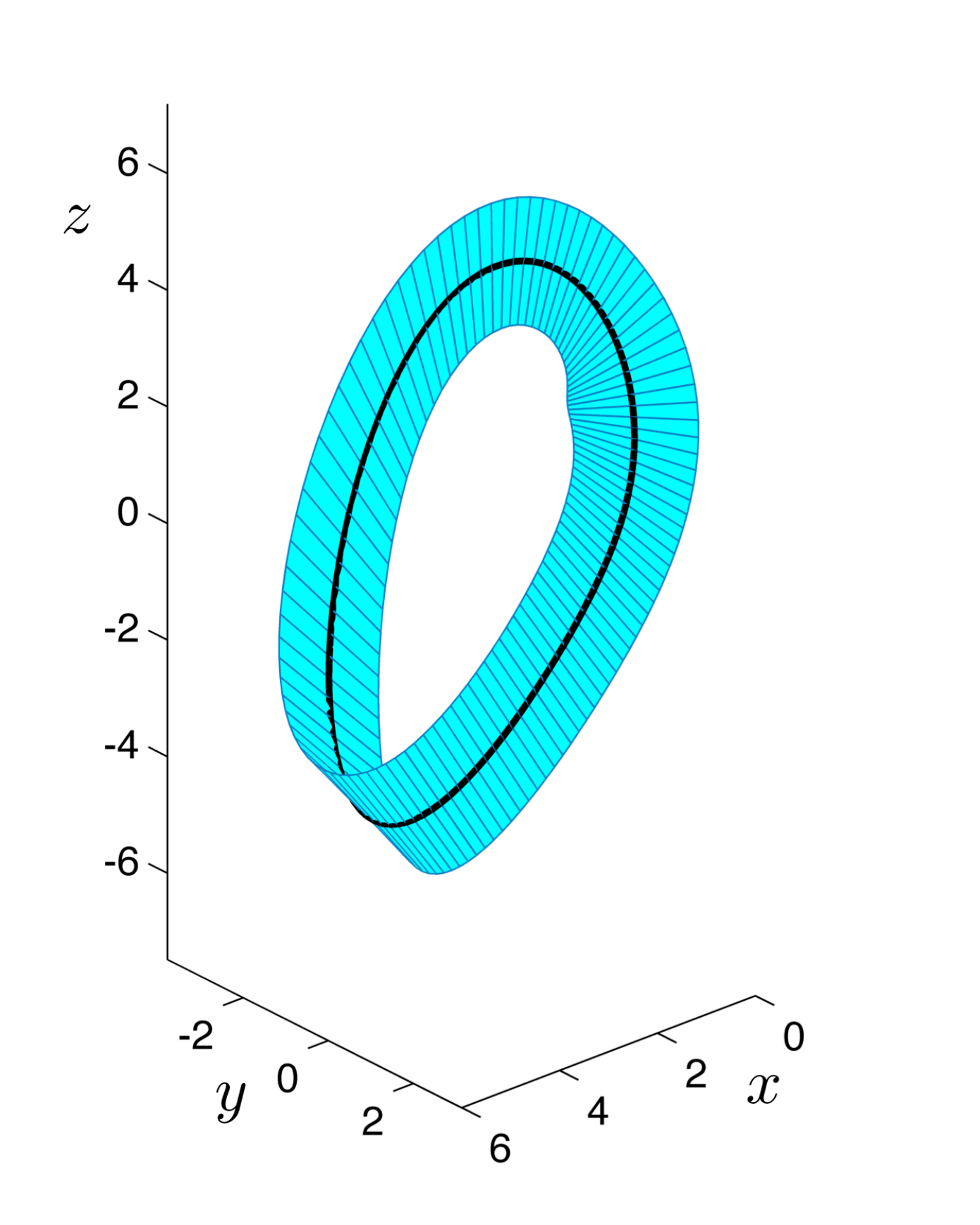}
}
\subfigure[]{
\includegraphics[width=0.47\textwidth]{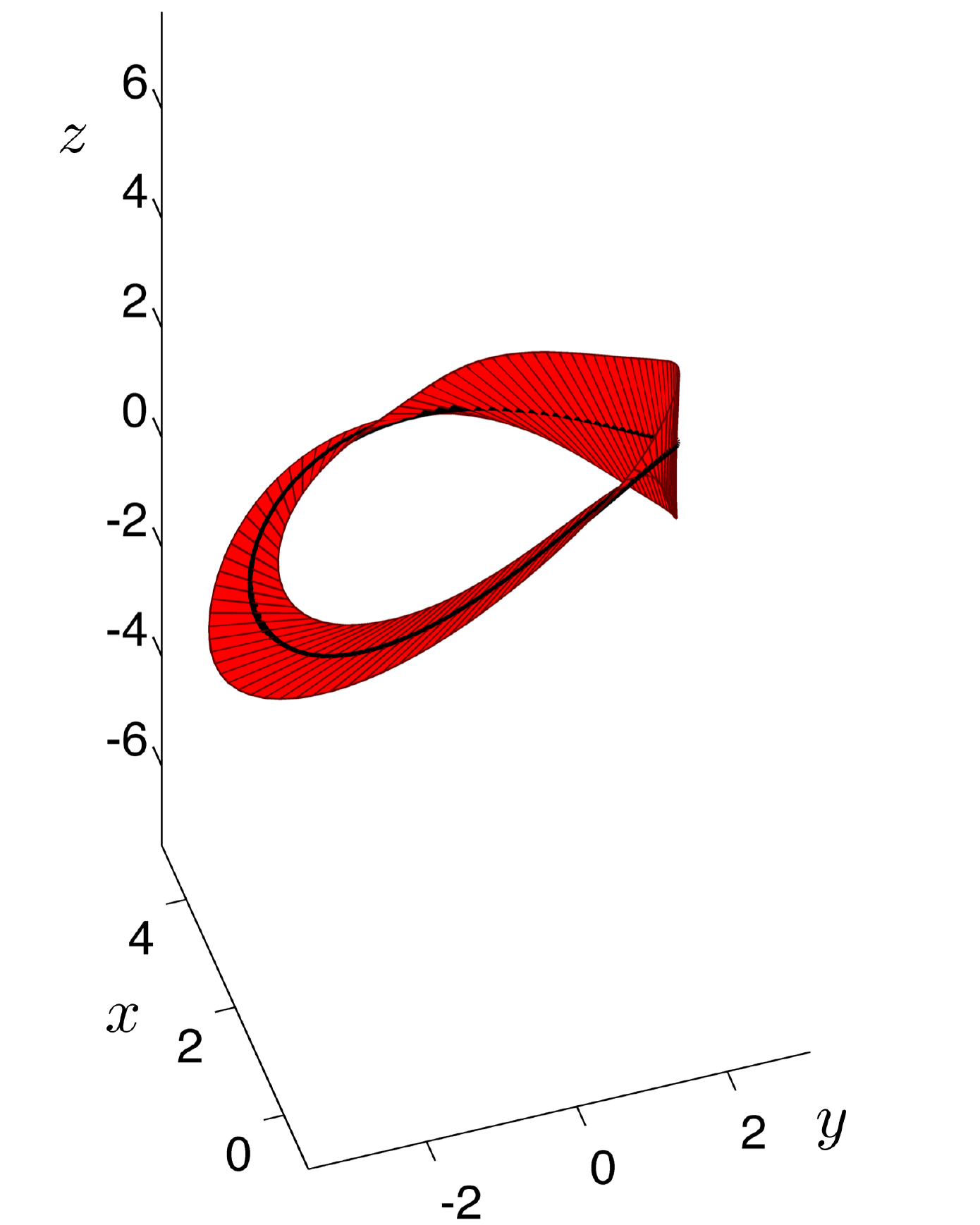}
}
\caption{{\small Stable (turquoise) and unstable (red) tangent bundles of a periodic orbit for the $\zeta^{3}$ model with negative Floquet multipliers}}
\label{fig:arneodo}
\end{figure}

\section{Acknowledgments}
We would like to thank Marcio Gameiro and Jason D. Mireles James for helpful discussions.

\section{Appendix}

Period and Fourier coefficients of $\bar\gamma_{1}$ and $\bar \gamma_{4}$.
\vspace{10pt}
\tiny

Solution $\#$ 1

$\bar \tau_{\gamma}=1.027854840752128$,\quad   $\bar \xi_{0}=\left[\begin{array}{c}  -4.606354666884038\\
  -4.606354666884038\\
  13.533127936581090   \end{array}\right]
$
$$
\begin{array}{r|r|r}
\bar\xi_{1} \hspace{60pt}&\bar \xi_{2}\hspace{60pt}& \bar\xi_{3}\hspace{60pt}\\
\hline
 \noalign{\vskip 2pt}
-2.457589444025310 - 0.734232230617879i &-0.840331595816763 - 0.280556868231764i &-0.244051806246999 - 0.079599377946577i\\
 -2.008759805775985 - 2.236534819812307i &-0.497327754727251 - 1.307931343042055i &-0.098076628971404 - 0.527159479549944i\\
  3.644641521462053 - 2.592383847330301i  &1.373048803380924 - 0.590924707028489i  &0.516674965708581 - 0.103578564294281i\\
  \noalign{\vskip 5pt}
\bar\xi_{4}  \hspace{60pt}&\bar \xi_{5}  \hspace{60pt}& \bar\xi_{6} \hspace{60pt}\\
\hline
 \noalign{\vskip 2pt}
  -0.066251105459624 - 0.023092599715315i &-0.017586386797219 - 0.007027398326009i &-0.004633846567037 - 0.002168050944621i\\
 -0.009785901431526 - 0.185087447977630i  &0.003892543961801 - 0.060779408215756i  &0.003318015096906 - 0.019163826327315i\\
  0.181193087615871 - 0.007305965402112i  &0.059967166034911 + 0.005022757182874i  &0.019005700003600 + 0.003642138482031i\\
   \noalign{\vskip 5pt}
\bar\xi_{7}  \hspace{60pt}&\bar \xi_{8}  \hspace{60pt}& \bar\xi_{9} \hspace{60pt}\\
\hline
 \noalign{\vskip 2pt}
 -0.001215481367863 - 0.000664841070850i &-0.000317263424301 - 0.000201467445482i &-0.000082326458994 - 0.000060305586234i\\
  0.001629398704351 - 0.005865931787398i  &0.000667978649296 - 0.001752989942270i  &0.000249451961957 - 0.000513234481766i\\
  0.005830102115005 + 0.001711097840891i  &0.001743196289261 + 0.000687999808866i  &0.000510298663577 + 0.000254394262682i\\
   \noalign{\vskip 5pt}
\bar\xi_{10}  \hspace{60pt}&\bar \xi_{11}  \hspace{60pt}& \bar\xi_{12} \hspace{60pt}\\
\hline
 \noalign{\vskip 2pt}
 -0.000021217174006 - 0.000017849897868i &-0.000005425393706 - 0.000005230883352i &-0.000001374889998 - 0.000001519316512i\\
  0.000087897665531 - 0.000147548597224i  &0.000029748123926 - 0.000041712327941i  &0.000009770046204 - 0.000011604812976i\\
  0.000146654700161 + 0.000089133299138i  &0.000041443384209 + 0.000030059169428i  &0.000011525197547 + 0.000009848310635i\\
 \noalign{\vskip 5pt}
\bar\xi_{13}  \hspace{60pt}&\bar \xi_{14}  \hspace{60pt}& \bar\xi_{15} \hspace{60pt}\\
\hline
 \noalign{\vskip 2pt}
 -0.000000344798445 - 0.000000437771341i &-0.000000085406182 - 0.000000125225422i &-0.000000020840071 - 0.000000035583174i\\
  0.000003134076053 - 0.000003177810272i  &0.000000986282445 - 0.000000856137980i  &0.000000305435097 - 0.000000226673423i\\
  0.000003154570627 + 0.000003153649381i  &0.000000849430424 + 0.000000991127939i  &0.000000224755146 + 0.000000306618231i\\
\end{array}
$$


Solution $\#$ 4

$\bar \tau_{\gamma}=0.683813590045746$,\quad $\bar \xi_{0}=\left[\begin{array}{c} -7.521252250993276
\\  -7.521252250993276\\
  22.399077327399255\\
  \end{array}\right]
$

$$
\begin{array}{c|c|c}
\bar\xi_{1} &\bar \xi_{2}& \bar\xi_{3}\\
\hline
 \noalign{\vskip 2pt}-1.246453092091490 - 1.262394967959499i &  -0.114587819240451 - 0.194828460289262i & -0.008589210614442 - 0.020708730881492i\\
 -0.117098081743200 + 1.579694736267260i & 0.015748373138453 + 0.188342422139680i & 0.002967722490543 + 0.020459531930793i\\
 -1.138858133842916 - 1.528289275687748i & -0.143528439361466 - 0.149121347835789i & -0.016769839350022 - 0.016159568619679i\\ \noalign{\vskip 5pt}
\bar\xi_{4} &\bar \xi_{5} & \bar\xi_{6}\\
\hline
 \noalign{\vskip 2pt} -0.000619073947482 - 0.001950094567689i & -0.000044052745410 - 0.000171023517657i & -0.000003073584415 - 0.000014085485779i\\
  0.000325236831105 + 0.001967969477674i & 0.000031364652698 + 0.000173341352082i & 0.000002859395743 + 0.000014272552631i\\
 -0.001814442573257 - 0.001764257857272i & -0.000188148978551 - 0.000185385380347i & -0.000018837629098 - 0.000018676999462i\\
   \noalign{\vskip 5pt}
 \bar \xi_{7} &\bar \xi_{8}&\bar \xi_{9}\\
\hline
 \noalign{\vskip 2pt}
 -0.000000208407126 - 0.000001089211660i & -0.000000013587818 - 0.000000078416175i & -0.000000000836854 - 0.000000005128148i\\
  0.000000251244910 + 0.000001104010427i & 0.000000021464593 + 0.000000079624828i & 0.000000001792309 + 0.000000005227527i\\
 -0.000001824392597 - 0.000001813988748i & -0.000000171368842 - 0.000000170686207i & -0.000000015658534 - 0.000000015616179i\\
  \noalign{\vskip 5pt}
  \bar\xi_{10} & \bar\xi_{11} & \bar\xi_{12}\\
\hline
 \noalign{\vskip 2pt}
 -0.000000000046957 - 0.000000000284715i & -0.000000000002182 - 0.000000000010252i & -0.000000000000052 + 0.000000000000357i\\
  0.000000000146750 + 0.000000000292804i & 0.000000000011806 + 0.000000000010900i  & 0.000000000000935 - 0.000000000000305i\\
 -0.000000001395354 - 0.000000001392990i & -0.000000000121501 - 0.000000000121396i & -0.000000000010352 - 0.000000000010351i\\
  \noalign{\vskip 5pt}
  \bar\xi_{13} & \bar\xi_{14} &\bar \xi_{15}\\
\hline
 \noalign{\vskip 2pt} 
 0.000000000000006 + 0.000000000000135i & 0.000000000000002 + 0.000000000000021i & 0.000000000000000 + 0.000000000000003i\\
  0.000000000000073 - 0.000000000000131i & 0.000000000000006 - 0.000000000000020i  & 0.000000000000000 - 0.000000000000003i\\
 -0.000000000000863 - 0.000000000000865i & -0.000000000000070 - 0.000000000000071i & -0.000000000000005 - 0.000000000000006i\\
\end{array}
$$

\normalsize
\bigskip
\noindent Numerical approximation $\bar R$ and even Fourier coefficients $\bar Q_{k}$.

\vspace{10 pt}
\tiny
Solution $\#$ 1
$$
\bar R=\left[
\begin{array}{rrr}
-10.508958375451483   &6.244108010218356  &-7.445538972862637\\
   1.367770562612481   &5.059391467543374 &-10.140640221489871\\
  -6.918853545877750   &5.863201994753524  &-8.217099758758689\\
\end{array}\right]
$$
$$\bar Q_{0}=\left[\begin{array}{rrr} 
1.411735844583484  &-0.999238898309471   &1.303728854375973\\
  -1.110555333911319  &0.141112697583195   &0.194620640182095\\
   0.843423174876108   &0.200767676367087  &-0.730409462001914\\
 \end{array}\right]
 $$ 
  $$
\bar Q_{2}=
$$
\vspace{-8pt}
$$
\left[\begin{array}{rrr}
-0.263681129667594 - 0.434148218933761i  &0.245502648963791 + 0.051118262858919i &-0.264146740709986 - 0.219962925571520i\\
  0.495149611356963 + 0.020091337974708i  &0.018943994759567 - 0.173000962596422i  &0.034736062134259 + 0.070721806450736i\\
 -0.125150133979381 + 0.155543962133145i  &0.072778675242146 + 0.276759663116221i  &0.234684069057806 - 0.137872983068902i\\
   \end{array}\right]
 $$
 $$
\bar Q_{4}= 
$$
\vspace{-8pt}
$$
\left[\begin{array}{rrr} 
  0.003223434126082 + 0.072191532567703i  &0.154600837185400 - 0.083425433638396i &-0.224496433432301 - 0.079860905176751i\\
  0.088047497733925 + 0.044110570228044i  &0.205199858834844 + 0.061631720475212i &-0.101564118363285 - 0.257855508650553i\\
 -0.138910366811394 + 0.089747331219474i &-0.038756214412790 + 0.227619738561727i  &0.277112301560736 - 0.128593935735273i\\
   \end{array}\right]
   $$
    $$
\bar Q_{6}= 
$$
\vspace{-8pt}
$$
\left[\begin{array}{rrr} 
 0.032229467843581 + 0.043355074683078i  &0.063509812104330 - 0.036923452979636i &-0.104043649831591 - 0.038640713988921i\\
 -0.000475098772244 + 0.078582831969487i  &0.122476382683268 + 0.065280023891177i &-0.036087360927645 - 0.192529574267658i\\
 -0.085890896774482 + 0.011537815519094i &-0.060779388764130 + 0.118929071081753i  &0.189091902999670 - 0.038256089963182i\\
 \end{array}\right]
   $$
    $$
\bar Q_{8}= 
$$
\vspace{-8pt}
$$
\left[\begin{array}{rrr} 
 0.014837677547204 + 0.016791777242045i  &0.023655160613775 - 0.012807406159434i &-0.039506811034992 - 0.015496956153149i\\
 -0.011244165701100 + 0.044396287876498i  &0.053040611505731 + 0.039952372292969i &-0.004186361886575 - 0.098878456196623i\\
 -0.042670724306854 - 0.009276118631273i &-0.039817529715657 + 0.050937744964250i &0.096680260325811 - 0.002950867660488i\\
 \end{array}\right]
   $$
    $$
\bar Q_{10}= 
$$
\vspace{-8pt}
$$
\left[\begin{array}{rrr} 
  0.005198646646053 + 0.006059114013686i  &0.008265722985980 - 0.003957671147668i &-0.013457029500778 - 0.005897110459635i\\
 -0.008342719368296 + 0.019119854569593i  &0.019900084741206 + 0.019629584406457i  &0.003372357081277 - 0.042680213775582i\\
 -0.018320486491551 - 0.008387560951806i &-0.019826346511011 + 0.019224271692316i  &0.042021880260296 + 0.004167728643705i\\
 \end{array}\right]
   $$
    $$
\bar Q_{12}= 
$$
\vspace{-8pt}
$$
\left[\begin{array}{rrr} 
  0.001642111476789 + 0.002113215989109i  &0.002770838660308 - 0.001130105417024i &-0.004317304940509 - 0.002176073388202i\\
 -0.004468250361663 + 0.007266315842478i  &0.006811683852990 + 0.008504503427600i  &0.003179105256374 - 0.016650157014229i\\
 -0.007054400616465 - 0.004581594079138i &-0.008590105865316 + 0.006619388137483i  &0.016483705220024 + 0.003472785816957i\\
  \end{array}\right]
   $$
    $$
\bar Q_{14}= 
$$
\vspace{-8pt}
$$
\left[\begin{array}{rrr} 
 0.000489747959995 + 0.000714963537264i  &0.000900858813617 - 0.000302058076912i &-0.001332830306692 - 0.000777133607668i\\
 -0.002038899547182 + 0.002555634111674i  &0.002173495066311 + 0.003390709655169i  &0.001809581469994 - 0.006067809090768i\\
 -0.002503668907675 - 0.002080001556342i &-0.003415963753540 + 0.002118659036427i  &0.006023316397811 + 0.001899418469947i\\
\end{array}\right]
   $$

Solution $\#$ 4

$$
\bar R=\left[
\begin{array}{rrr}
-10.103827000749006 &5.011512150268070  &-4.181592133228406\\
  2.108771563239242 &-0.623931925418962  &0.486619976897008\\
  -6.292527840128125 &3.486887629270139  &-2.938907740498710\\
\end{array}\right]
$$
$$\bar Q_{0}=\left[\begin{array}{rrr} 
   0.865350358013670  &-0.542407880588934   &0.461699549706412\\
  -0.413891831683466   &0.086095506914414  &-0.062238564323011\\
   0.495177534990071  &-0.049624165086980   &0.025622831143080\\
\end{array}\right]
$$

 $$
\bar Q_{2}=
$$
\vspace{-8pt}
$$
\left[\begin{array}{rrr} 
  0.043888342026888 - 0.018225551633217i  &0.237009816130334 - 0.161848354458699i &-0.187509342015513 - 0.182145921801054i\\
  0.184261504574051 + 0.121001510966144i  &0.327547824090844 - 0.006620885935044i  &0.028142805119932 - 0.301161573224061i\\
 -0.196713672633492 + 0.151158740280049i  &0.033125259013010 + 0.469308310958510i  &0.383336675401313 - 0.015987641512644i\\
  \end{array}\right]
 $$
 $$
\bar Q_{4}= 
$$
\vspace{-8pt}
$$
\left[\begin{array}{rrr} 
  0.020354746043569 + 0.006072390169969i & 0.031468524011727 - 0.044530642983650i &-0.037967521528517 - 0.020350032723162i\\
  0.019156656540045 + 0.040757474867593i  &0.108300555430671 + 0.012024500455431i  &0.003607621081908 - 0.088848037842172i\\
 -0.042700076948069 + 0.023165829211282i &-0.006988160820195 + 0.106913456324494i & 0.087287750910915 + 0.001272272562888i\\
    \end{array}\right]
    $$
$$
\bar Q_{6}= 
$$
\vspace{-8pt}
$$
\left[\begin{array}{rrr} 
  0.002761640132931 + 0.000187559216871i &  0.002550968100164 - 0.005879343921945i &-0.004808213884619 - 0.001470031764531i\\
  0.003017836182538 + 0.007295810537329i & 0.018305847192402 + 0.001100730065600i &-0.000373601315687 - 0.014656539554954i\\
 -0.007076735073531 + 0.002990041055165i & -0.001327830255369 + 0.017843392559588i & 0.014369458507765 - 0.000147876219115i\\
    \end{array}\right]
 $$
 
 $$
\bar Q_{8}= $$
\vspace{-8pt}
$$
\left[\begin{array}{rrr} 
  0.000289533560511 - 0.000023944541333i & 0.000167141543813 - 0.000637470930420i &-0.000509911668009 - 0.000073299205025i\\
  0.000441108794910 + 0.000976089854840i & 0.002466959960129 - 0.000020947005419i &-0.000203588657299 - 0.001946883870274i\\
 -0.000969488479052 + 0.000417372380704i &-0.000027259876302 + 0.002448108770303i & 0.001937426651095 - 0.000166408614554i\\
    \end{array}\right]
 $$
 $$
\bar Q_{10}= 1.1e-03 *
$$
\vspace{-8pt}
$$
\left[\begin{array}{rrr} 
  0.027798782570345 - 0.005964538385953i  &0.007649736579992 - 0.063378060920909i &-0.049748597822010 - 0.000258862545890i\\
  0.060031396019271 + 0.114573660552919i & 0.294930056067831 - 0.027358497808776i &-0.045190737024386 - 0.229328952532320i\\
 -0.115071668553051 + 0.057995199816582i  &0.023085665062373 + 0.294833391391544i & 0.229589287166289 - 0.041943841171559i\\
    \end{array}\right] 
 $$
 $$
\bar Q_{12}=    1.1e-04 *
$$
\vspace{-8pt}
$$
\left[\begin{array}{rrr} 
  0.025222741873242 - 0.008833955070603i &-0.000684171241317 - 0.059552377024599i &-0.045934587657163 + 0.005873935636186i\\
  0.077200510619449 + 0.122891998519810i & 0.324298385293782 - 0.061987569005484i &-0.075398660882460 - 0.248296827479228i\\
 -0.123626123118424 + 0.075719660830317i & 0.058531951875591 + 0.324753469250377i & 0.248928813781281 - 0.072775716802859i\\
  \end{array}\right]  
  $$
   $$
\bar Q_{14}=    1.1e-05 *
$$
\vspace{-8pt}
$$
\left[\begin{array}{rrr} 
  0.021844676611758 - 0.010885662915616i &-0.007613608020911 - 0.053546469056524i &-0.040603455774486 + 0.010691275157818i\\
  0.093733134787201 + 0.122596846519192i & 0.334057192499540 - 0.100903036175184i &-0.106940875623844 - 0.251662436543305i\\
 -0.123324223393918 + 0.092568245260174i & 0.097987409026114 + 0.334717932778523i & 0.252425413480351 - 0.104748352612175i\\
   \end{array}\right]  
  $$
\normalsize

\bigskip

\noindent Enclosure of the spectrum and eigenvectors of $R$.
\tiny
\\

\noindent
Solution $\#$ 1
$$
\begin{array}{r|r|r|r}
& {\it Stable} & & {\it Unstable}\\
\hline \noalign{\vskip 2pt}
{\it E.values}&-14.295385513026014&  -0.000000000000342  & 0.628718846359581\\
 \noalign{\vskip 2pt}
 & -1.304013849063401 &  0.330244752093107  & 0.376869068070140\\
{\it E.vectors}  &-0.455394137737842  & 1.503215970221508   &1.529903371837170\\
  &-1.045066534133045  & 0.794531403146519  & 0.719281153912157\\
 \noalign{\vskip 2pt} {\it Rad}\  10^{-3}\cdot&0.068012486147407  & 0.933273952985148 &  0.951085304085387 \\
  \end{array}
  $$
  Solution $\#$ 2
$$
\begin{array}{r|r|r|r}
& {\it Stable} & & {\it Unstable}\\
\hline \noalign{\vskip 2pt}
{\it E.values}&-14.217489884945432  &-0.000000000000372&   0.550823218279069\\
 \noalign{\vskip 2pt}
  &-1.311412833044274  & 0.313018700843905  & 0.362984001981858\\
{\it E.vectors} & -0.418001967490342  & 1.491825566190409  &1.522008146749871\\
  &-1.051413684760184 &  0.822481472729119  &0.742787867114333\\
  \noalign{\vskip 2pt} {\it Rad}\  10^{-3}\cdot& 0.028142829394895  & 0.366070961579701  & 0.3733964559921\\
 \end{array}
$$
Solution $\#$ 3
$$
\begin{array}{r|r|r|r}
& {\it Stable} & & {\it Unstable}\\
\hline \noalign{\vskip 2pt}
{\it E.values} & -13.962049368058929 & -0.000000000000126  & 0.295382701392358\\
 \noalign{\vskip 2pt}  & 1.347327907101522 &  0.210153254038267 & -0.271285496065970\\
{\it E.vectors} &   0.192884294913700  & 1.398977370916435 &  -1.447969583732432\\
 &  1.071215739018561 &  0.999348750677595 & -0.910927145390874\\
 \noalign{\vskip 2pt} {\it Rad}\  10^{-4}\cdot&0.027747858117877 &  0.357120423613300  & 0.365316866760002\\
 \end{array}
$$
Solution $\#$ 4
$$
\begin{array}{r|r|r|r}
& {\it Stable} & & {\it Unstable}\\
\hline \noalign{\vskip 2pt}
{\it E.values} &
-13.721015009104903  &-0.000000000000309 &  0.054348342438550\\
  \noalign{\vskip 2pt}  &  1.439298428490600   &0.051023279023768 & -0.128543563503133\\
{\it E.vectors} &  -0.266150110668696  & 1.168676456309931 & -1.252767532277167\\
&  0.926058395748094  & 1.277337843119246 & -1.189138369725781\\
 \noalign{\vskip 2pt} {\it Rad}\  10^{-4}\cdot &0.025442623394767  & 0.486579977382052  & 0.524845634189975\\
  \end{array}
$$
Solution $\#$ 5
$$
\begin{array}{r|r|r|r}
& {\it Stable} & & {\it Unstable}\\
\hline \noalign{\vskip 2pt}
{\it E.values} &
-13.701329239339196 & -0.000000000000487 &  0.034662572673008\\
   \noalign{\vskip 2pt}  &   1.451175793211715 &  0.020133220407690  &-0.100604795866982\\
{\it E.vectors} &  -0.333782791906082 &  1.115780576244332  &-1.206723375086350\\
 &  0.884690830190827 &  1.324623855708436  &-1.238425359506485\\
 \noalign{\vskip 2pt} {\it Rad}\  10^{-3}\cdot &0.009720262854618 &  0.336965856255784  & 0.333681919907764\\
  \end{array}
$$

\normalsize

\bibliographystyle{unsrt}
\bibliography{papers}

\end{document}